\documentclass[11pt,english]{amsart}

\usepackage[english]{babel}
\usepackage[T1]{fontenc}
\usepackage{lmodern}

\theoremstyle{plain}
\newtheorem{theo}{Theorem}[section]
\newtheorem{lemm}[theo]{Lemma}
\newtheorem{prop}[theo]{Proposition}
\newtheorem{coro}[theo]{Corollary}

\theoremstyle{definition}
\newtheorem{defi}[theo]{Definition}

\theoremstyle{remark}
\newtheorem{rema}[theo]{Remark}

\usepackage{hyperref}
\usepackage{microtype}
\usepackage{enumerate}
\usepackage{graphicx}
\usepackage{color}
\usepackage{fullpage}

\usepackage{bm}
\usepackage{mathtools}

\title[Analytic continuation of Liouville Theory: pencil region]{Path integral approach to analytic continuation of Liouville Theory: the pencil region}
\author{Yichao Huang}
\address{Exactum C327, P.O. Box 68 (Pietari Kalmin katu 5), FI-00014 University of Helsinki}
\email{Yichao.Huang@helsinki.fi}
\thanks{Research supported by Institut Mittag-Leffler postdoctoral fellowship and ERC-grant QFPROBA}
\date{\today}

\begin{document}

\begin{abstract}
We study the problem of analytic continuation of Liouville Conformal Field Theory using the probabilistic approach of David, Kupiainen, Rhodes and Vargas \cite{david2016liouville} based on the theory of Gaussian Multiplicative Chaos. The key idea is to apply stochastic calculus techniques to some Brownian motions associated to the Gaussian Free Field. We strengthen the results in \cite{kupiainen2017integrability} and are able to provide rigorous justification of analytic continuation of Liouville Theory in an infinite trianglar region, which we call the pencil region. This is in accordance with the physics literature \cite{harlow2011analytic}. Along the lines, we develop new tools and estimates of Gaussian measures and Liouville correlation functions using this probabilistic approach.

\end{abstract}
\keywords{Liouville Conformal Field Theory, Liouville Correlation Functions, Analytic Continuation, Gaussian Free Field, Gaussian Multiplicative Chaos.}
\maketitle
\setcounter{tocdepth}{1}
\tableofcontents

\section{Introduction}
The path integral formalism of Liouville Conformal Field Theory (LCFT hereafter) was proposed in the seminal paper of Polyakov \cite{polyakov1981quantum} and can been seen as a probabilistic theory of $2d$ Riemannian metrics. In a series of recent works \cite{david2016liouville,kupiainen2015local,kupiainen2017integrability} by David, Kupiainen, Rhodes and Vargas, a rigorous mathematical construction of Polyakov's path integral formalism is carried out, which defines the Liouville correlation functions in a probabilistic setting, and consequently provides rigorous mathematical proofs of fundamental formulas in Conformal Field Theory such as the BPZ equations \cite{belavin1984infinite} and the DOZZ formula \cite{dorn1992correlation,zamolodchikov1996conformal} on the 3-point structure constant. Their construction is based on Gaussian Free Field (GFF hereafter) and its exponential which is defined using the theory of Gaussian Multiplicative Chaos (GMC hereafter) of Kahane \cite{kahane1985chaos}. The goal of this article is to investigate extension of this path integral construction to Liouville correlation functions with complex parameters: this problem is known in the physics literature as analytic continuation of Liouville theory \cite{harlow2011analytic}. We refer the reader to the introductory reviews \cite{vargas2017lecture,kupiainen2018dozz,harlow2011analytic,ribault2014conformal} for background and references both in mathematics and physics on this subject.

\subsection{Main result}
The goal of this paper is to prove that the path integral formalism proposed in \cite{david2016liouville} directly defines a natual analytic function on a much larger region for general $n$-point Liouville correlation functions. In the language of LCFT, our main statement is the following (definitions and a precise statement will be given later):
\begin{theo}[Main theorem]\label{th:MainTheorem}
Let $\alpha_i$ be real parameters satisfying the Seiberg bounds
\begin{equation}
\forall i,\quad \alpha_i<Q\quad\text{and}\quad \sum\limits_{i}\alpha_i>2Q
\end{equation}
on the Riemann sphere. The $n$-point Liouville correlation function with $\alpha_j\in\mathbb{R}$, $z_j\in\mathbb{C}$ and $\mu>0$ defined by the path integral formalism in \cite{david2016liouville} on the Riemann sphere
\begin{equation}
\left<\prod_{j=1}^{n}V_{\alpha_j}(z_j)\right>\coloneqq\mathbb{E}\left[\int_{\mathbb{R}}e^{(\sum_j\alpha_j-2Q)c}\prod_j e^{\alpha_j X(z_j)}e^{-\mu e^{\gamma c}\int_{\mathbb{C}}e^{\gamma X}}dc\right]
\end{equation}
also defines a natural analytic continuation to complex valued insertions
\begin{equation*}
\alpha_j+i\beta_j\in\mathbb{C}
\end{equation*}
in the region (which we refer to as the pencil region)
\begin{equation}
\mathcal{R}\coloneqq\cap_j\{|\beta_j|<Q-\alpha_j\}.
\end{equation}
More precisely, the n-point correlation function with complex parameters $(\alpha_j+i\beta_j)_{1\leq j\leq n}\in(\mathbb{C})^n$:
\begin{equation}
\left<\prod_{j=1}^{n}V_{\alpha_j+i\beta_j}(z_j)\right>\coloneqq\mathbb{E}\left[\int_{\mathbb{R}}e^{(\sum_j(\alpha_j+i\beta_j)-2Q)c}\prod_j e^{(\alpha_j+i\beta_j)X(z_j)}e^{-\mu e^{\gamma c}\int_{\mathbb{C}}e^{\gamma X}}dc\right]
\end{equation}
can be defined using the standard regularization procedure similar to the one used in the \cite{david2016liouville,kupiainen2017integrability} and is analytic in all $\alpha_j+i\beta_j$ in the pencil region $\mathcal{R}$.
\end{theo}
The novelty in this paper is to apply various techniques from stochastic calculus to LCFT in order to obtain non-trivial estimates in the case with complex parameters. In particular, two different approaches are proposed to study the convergence in the region $\mathcal{R}$: either by applying Itô calculus or by using the renewal theory to a certain Browian motion associated with the vertex operators.

\subsection{Relation to the work of Kupiainen-Rhodes-Vargas}
Analytic continuation of LCFT via GMC approach was first studied in \cite[Section~4]{kupiainen2017integrability} and was essential for rigorously proving the DOZZ formula using probabilistic approach. The authors of \cite{kupiainen2017integrability} applied radial decomposition (see Section~\ref{sec:Radial} below for a brief review) to some GFF locally and obtained an exponential convergence of Liouville correlation functions along a subsequence of times. As a result, they were able to define analytic continuation of Liouville correlation functions in a small neighborhood near each real parameter $\alpha_j$.

Our work strengthens their result. We identify explicitly a large region $\mathcal{R}$ for which Liouville analytic continuation holds within this probabilistic framework; this is coherent with the physics literature \cite{harlow2011analytic} by Harlow-Maltz-Witten on analytic continuation of Liouville theory, although we are not able to identify our pencil region with the physical region of \cite{harlow2011analytic}. Our method yields in some particular cases explicit formulas and precise estimations on Liouville correlation functions, and we expect to recover or obtain exact formulas using methods based on these new observations.

\subsection{Motivations from conformal bootstrap program}
The problem of investigating analytic continuation of Liouville correlation functions in the parameters $\{\alpha_i\}$ fits naturally within the ongoing program in constructive LCFT of Kupiainen-Rhodes-Vargas, of which the goal is to unify two different approaches to LCFT in physics: the path-integral approach and the conformal bootstrap approach. Without going into too much details, let us cite for example one formula from the conformal bootstrap picture that remains mathematically conjectural, see \cite{teschner2001liouville} for example for a detailed discussion:
\begin{equation}
\left<V_{\alpha_1}(z)V_{\alpha_2}(0)V_{\alpha_3}(1)V_{\alpha_4}(\infty)\right>=\int_{\mathbb{R}_+}C_\gamma(\alpha_1,\alpha_2,Q+iP)C_\gamma(Q-iP,\alpha_3,\alpha_4)\left|\mathcal{F}_{Q+iP,\{\alpha_i\}}(z)\right|^2dP
\end{equation}
where $C_\gamma(\alpha_1,\alpha_2,Q+iP)$ is the three-point correlation function: its value is explicitly known since the DOZZ formula. On the other hand, $\mathcal{F}_{\alpha,\{\alpha_i\}}(z)$ are called universal conformal blocks and are explicit meromorphic functions depending only on $Q+iP$ and $\{\alpha_i\}$. The line $Q+i\mathbb{R}_+$ over which we integrate should correspond to the spectrum of LCFT.

It is tempting to give probabilitic interpretation to this decomposition over the spectrum $Q+iP$ for $P\geq 0$, and to deduce for instance an probabilistic expression for computing the conformal blocks that appear above. Recent progress has been made towards this goal (see \cite{kupiainen2016constructive,baverez2018modular,baverez2018fusion}) and the current article is motivated by defining the vertex operators over the critical line, i.e. with parameter $Q+iP$ for $P\in\mathbb{R}$. In other words, we investigate the question of whether the general $n$-point Liouville correlation function $C_\gamma(\{\alpha_i\},Q+iP)$ can be defined directly using the probabilistic approach of \cite{david2016liouville}. While this idea is directly inspired by the physicists \cite{harlow2011analytic} but based on probabilistic constructions \cite{david2016liouville,kupiainen2017integrability}, we are not able to get very close to the critical line at this moment: the regularization procedure used in the current version diverges when we get out of of the pencil region that we identify explicitly, where same kind of phenomenon has been observed in \cite{harlow2011analytic}. We intend to continue the study of this problem in an upcoming work.

\subsection{Acknowledgements}
We gratefully appreciate stimulating advices from Antti Kupiainen, Rémi Rhodes and Vincent Vargas. We are indebted to Rémi Rhodes in particular for communicating some crucial ideas. We also thank Guillaume Baverez, Linxiao Chen and Joona Oikarinen for discussions. The work is finished at the Isaac Newton Institute during the program \emph{Scaling limits, rough paths, quantum field theory} and started at the Institut Mittag-Leffler during the program \emph{Fractal geometry and dynamics}: we wish to thank both institutes for their very warm hospitality.

\section{Preliminaries}\label{sec:Preliminaries}

\emph{Notations.} Throughout this article, $X$ will denote a Gaussian random field in some region of the Riemann sphere $\widehat{\mathbb{C}}=\mathbb{C}\cup\{\infty\}$. Parameters $\alpha_j+i\beta_j\in\mathbb{C}$ will be complex numbers, and $\{z_i\}$ are points on the Riemann sphere $\widehat{\mathbb{C}}$. We denote by $\bm{z}$, $\bm{\alpha+i\beta}$, $\bm{r}$ some $n$-dimensional vectors of resp. $z,\alpha+i\beta,r$. We also use classical notations $\gamma\in(0,2)$, $Q=\frac{2}{\gamma}+\frac{\gamma}{2}$ and $\mu>0$ to denote parameters in the Liouville action as in \cite{david2016liouville}.
\subsection{Geometric setup}\label{sec:metric}
Throughout the rest of this article we will use a fixed metric $\mathbf{g}(z)d^2 z$ on the Riemann sphere $\widehat{\mathbb{C}}=\mathbb{C}\cup\{\infty\}$ of the form
\begin{equation}
\mathbf{g}(z)\coloneqq |z|_+^{-4}
\end{equation}
where $|z|_+=|z|\vee 1$. This metric has scalar curvature
\begin{equation}
R_\mathbf{g}(z)=-4\mathbf{g}^{-1}\partial_z\partial_{\overline{z}}\ln\mathbf{g}(z)=4\nu
\end{equation}
with $\nu$ the uniforme measure on the circle $\partial B(0,1)$.
\subsection{Gaussian Free Field}\label{sec:GFFSetup}
The GFF (for mathematical backgrounds, see \cite{sheffield2007gaussian,dubedat2009sle,david2016liouville}) with zero average in the uniforme measure on the circle $\nu$ has covariance kernel
\begin{equation}
K(x,y)\coloneqq\mathbb{E}[X(x)X(y)]=-\ln|x-y|+\ln|x|_{+}+\ln|y|_{+}
\end{equation}
for $x,y\in\mathbb{C}\cup\{\infty\}$. Notice that inside the unit disk $\mathbb{D}=B(0,1)$,
\begin{equation}
K(x,y)=K_{\mathbb{D}}(x,y)\coloneqq-\ln|x-y|
\end{equation}
where $K_{\mathbb{D}}$ is the (Neumann boundary condition) Green function inside the unit disk $\mathbb{D}$.

\subsection{Local radial decomposition}\label{sec:Radial}
Let $X$ be the centered log-correlated Gaussian field on the unit disk $\mathbb{D}=B(0,1)$ of covariance kernel
\begin{equation}\label{eq:G_D}
K_\mathbb{D}(x,y)\coloneqq\mathbb{E}[X(x)X(y)]=\ln\frac{1}{|x-y|}.
\end{equation}
Recall that $X(x)$ can be defined as a distribution (in the sense of Schwartz) and admits the following decomposition:
\begin{lemm}[Radial decomposition for log-correlated Gaussian field]\label{lem:Radial}
For every $x\in\mathbb{D}\backslash\{0\}$, $X(x)$ can be written as
\begin{equation}
X(x)=X_{|x|}(0)+N(x)
\end{equation}
where $X_{|x|}(0)$ is the circle-average on center $0$ and radius $|x|$ defined for $r>0$ as
\begin{equation}
X_{r}(0)=\frac{1}{2\pi}\int_{0}^{2\pi}X(r e^{i\theta})d\theta
\end{equation}
and $N(x)$ is the lateral noise distribution defined as
\begin{equation}
N(x)=X(x)-X_{|x|}(0).
\end{equation}
We have the following properties:
\begin{enumerate}
	\item The Gaussian fields $\{X_r(0), 0<r<1\}$ and $\{N(x), x\in\mathbb{D}\backslash\{0\}\}$ are independent.
	\item The process $B_t\coloneqq\{X_{e^{-t}}(0)\}$ is a standard Brownian motion.
\end{enumerate}
\end{lemm}
This decomposition is well-known in the literature: for example see \cite{sheffield2007gaussian,duplantier2014liouville,kupiainen2017integrability}. One can also verify this lemma directly on the specific kernel $K_\mathbb{D}$ of equation~\eqref{eq:G_D} by calculating explicitly the covariance of each Gaussian field.
\begin{rema}[Independence property]\label{rem:Independence}
We record here a useful property of the radial decomposition. On can extend the radial decomposition procedure to balls of the form $B(z,1)$ for $z\in\mathbb{C}$ by conformal mapping and in particular, the process $\widetilde{B}_t\coloneqq\{X_{e^{-t}}(z)-X_{0}(z)\}$ is also a Brownian motion starting at $0$. Then for disjoint balls $B(z_j,1)\subset B(0,1)^{c}$, these Brownian motions are mutually independent and independent of the $\sigma$-algebra generated by $\{X(z);z\notin\cup_j B(z_j,1)\}$.
\end{rema}

\subsection{Gaussian Multiplicative Chaos}\label{sec:Chaos}
The study of GMC measures started with the seminal work \cite{kahane1985chaos} of Kahane. The theory of GMC allows one to define mathematically exponentials of log-correlated Gaussian field in any dimension and in particular, the exponential of the GFF above in $2d$. We refer the reader to \cite{kahane1985chaos,robert2010gaussian,berestycki2017elementary} for more materials on the definition and convergence of GMC measure and \cite{rhodes2014gaussian} for more applications; let us briefly recall the idea and gather some notations here (we restrict ourselves to minimal setting, namely dimension $2$ and the metric $\mathbf{g}$):
\begin{defi}
Let $X$ be a $\log$-correlated field on a subdomain $\Omega\subset\mathbb{R}^2$ equipped with the metric $\mathbf{g}$. We can define the GMC measure associated to $X$ with parameter $\gamma\in(0,2)$,
\begin{equation}
M_{\gamma}(d^2z)\coloneqq e^{\gamma X(z)}\mathbf{g}(z)d^2z
\end{equation}
to be the limit of the family of random measures
\begin{equation}
M_{\gamma,\epsilon}(d^2z)\coloneqq e^{\gamma X_\epsilon(z)-\frac{\gamma^2}{2}\mathbb{E}[X_\epsilon(z)^2]}\mathbf{g}(z)d^2z
\end{equation}
as $\epsilon$ goes to $0$. Here, $X_\epsilon(z)$ denotes a regularization by a smooth mollifier of the field $X$ (another common regularization is by circle average similar to the previous section).
\end{defi}
Applying the radial decomposition to the GMC measures, we have the following:
\begin{prop}[Radial decomposition for GMC measures]
Let $X$ be the GFF as in Lemma~\ref{lem:Radial} and consider the domain $\Omega=\mathbb{D}\backslash B(0,e^{-t})$. The GMC random measure $M_\gamma(\Omega)$ associated to $X$ on the domain $\Omega$ has the following equivalent expression
\begin{equation}
M_\gamma(\Omega)\overset{(law)}{=}\int_{0}^{t}e^{\gamma(B_s-Qs)}Z_sds
\end{equation}
where
\begin{equation}
Z_s=\frac{1}{2\pi}\int_{0}^{2\pi}e^{\gamma Y(s,\theta)-\frac{\gamma^2}{2}\mathbb{E}[Y(s,\theta)^2]}d\theta
\end{equation}
with $Y(s,\theta)$ a Gaussian field (seen as a distribution) independent of $B_t$ with covariance kernel
\begin{equation}\label{eq:NoiseCorr}
\mathbb{E}[Y(s,\theta)Y(t,\theta')]=\ln\frac{e^{-s}\vee e^{-t}}{|e^{-s}e^{i\theta}-e^{-t}e^{i\theta'}|}.
\end{equation}
\end{prop}
\begin{proof}
For more details and discussions, see \cite[Section 2.7]{kupiainen2017integrability}.
\end{proof}

\subsection{Liouville correlation functions}
One main feature in the path integral formalism of LCFT as defined in \cite{david2016liouville} is that one can express Liouville correlation functions by means of GMC measures with $\log$-singularities. More precisely, by a change of variables, one can express the $n$-point correlation functions (with real parameters $\bm{\alpha}$) in the metric $\mathbf{g}$ (defined as in Section~\ref{sec:metric}) on the Riemann sphere $\mathbb{C}\cup\{\infty\}$ in the following manner:
\begin{equation}
\left<\prod_{j=1}^{n}V_{\alpha_j}(z_j)\right>=\frac{2}{\gamma}\mu^{-s}\Gamma(s)\prod_{k<l}\frac{1}{|z_k-z_l|^{\alpha_k\alpha_l}}\mathbb{E}\left[\left(\int_{\mathbb{C}}F(x,\mathbf{z})M_{\gamma}(d^2 x)\right)^{-s}\right]
\end{equation}
with
\begin{equation}
s=\frac{\sum_j\alpha_j-2Q}{\gamma}
\end{equation}
and
\begin{equation}
F(x,\mathbf{z})=\prod_{j=1}^{n}\left(\frac{|x|_+}{|x-z_j|}\right)^{\gamma\alpha_j}.
\end{equation}
In particular, the correlation function is well-defined if one can make sense of the negative moment and shows that it is positive: this can be done using GMC techniques (see \cite{david2016liouville,huang2018liouville}). One sufficient condition is known as the Seiberg bound \cite{seiberg1990notes}:
\begin{equation}
\forall i,\quad \alpha_i<Q\quad\text{and}\quad \sum\limits_{i}\alpha_i>2Q.
\end{equation}
We refer the reader to \cite{david2016liouville,kupiainen2017integrability} for proof and details on this expression. In the following, we will consequently study the analyticity of moments with negative real parts of GMC measure by the same regularization procedure as in \cite{kupiainen2017integrability}. Namely, we will study the convergence of the regularized $n$-point negative GMC measure:
\begin{equation}
G(\bm{\alpha+i\beta},\mathbf{z};\mathbf{t})\coloneqq\mathbb{E}\left[\prod_{j=1}^{n}e^{(\alpha_j+i\beta_j)X_{r_j}(z_j)-\frac{(\alpha_j+i\beta_j)^2}{2}t_j}M_{\gamma}(C_\mathbf{t})^{-s}\right]
\end{equation}
where $r_j=e^{-t_j}$, $C_{\mathbf{t}}$ is the regularized complex plane
\begin{equation}
C_{\mathbf{t}}\coloneqq \mathbb{C}-\bigcup_{j}B(z_j,e^{-t_j}),
\end{equation}
and the $r_j$-regularization procedure is the one described in Subsection~\ref{sec:Chaos}. Consequence of Remark~\ref{rem:Independence}, we identify $X_{r_j}(z_j)$ as independent Brownian motions and denote them by $B_j(t_j)$. Notice that the function $G(\bm{\alpha+i\beta},\mathbf{z};\mathbf{t})$ is complex differentiable in all components of $\mathbf{t}$, hence defines an multivariate entire function in the $\bm{\alpha}$.

Our goal is to establish conditions on local uniform convergence of $G_{\mathbf{t}}$ as $\mathbf{t}$ goes to infinity: this will yield analyticity of the limit function $G(\bm{\alpha+i\beta};\mathbf{z})$.
\subsection{Freezing estimate}
We use frequently in this paper an estimate on integrals of GMC measure with singularities known as the freezing estimate in the literature \cite{fyodorov2008freezing}:
\begin{lemm}\label{lem:Freezing}
For $\alpha>Q$ and $p>0$,
\begin{equation}
\mathbb{E}\left[\left(\int_{|x|>\epsilon}\frac{1}{|x|^{\gamma\alpha}}M_\gamma(d^2x)\right)^{-p}\right]\leq C\epsilon^{\frac{1}{2}(\alpha-Q)^2}
\end{equation}
and if $\mu>0$,
\begin{equation}
\mathbb{E}\left[\exp\left(-\mu\int_{|x|>\epsilon}\frac{1}{|x|^{\gamma\alpha}}M_\gamma(d^2x)\right)\right]\leq C\epsilon^{\frac{1}{2}(\alpha-Q)^2}.
\end{equation}
with some constant $C$ locally uniform in $\alpha$ and independent of $\epsilon$ when $\epsilon$ is small enough.
\end{lemm}
\begin{proof}
See \cite[Section~6]{kupiainen2015local}.
\end{proof}
In Appendix~\ref{app:freezing} we provide a variant that slightly generalizes the above estimate.

\subsection{Stopping time of drifted Brownian motion}
Let $\alpha<Q$ and consider a negatively drifted Brownian motion $B_t-(Q-\alpha)t$. We define the following stopping times for $n\in\mathbb{N}$:
\begin{equation}
T_n=\inf\{s:B_s-(Q-\alpha)s=-n\}.
\end{equation}
Recall several elementary facts that we will use in the following.
\begin{prop}[Elementary facts on stopping time of drifted Brownian motion]\label{prop:renewalprop}
We recall some basic facts on drifted Brownian motion:
\begin{enumerate}
	\item The law of $T_1$ follows an inverse Gaussian with parameter $IG((Q-\alpha)^{-1},1)$, i.e. its probability density function is
	\begin{equation}\label{eq:T1density}
	\left(\frac{1}{2\pi x^3}\right)^{\frac{1}{2}}\exp\left(-\frac{((Q-\alpha)x-1)^2}{2x}\right)dx.
	\end{equation}
	In particular, from the exponential tail of the above density function,
	\begin{equation}
	\mathbb{E}\left[e^{\frac{\beta^2}{2}T_1}\right]<\infty
	\end{equation}
	if and only if $|\beta|<Q-\alpha$.
	\item The sequence $(T_{i+1}-T_{i})_{i\in\mathbb{N}}$ is an i.i.d. sequence distributed as $T_1$.
	\item Let $t>0$. Define the residual time of the sequence $(T_{i})_{i\in\mathbb{N}}$ at time $t$ as
	\begin{equation}
	R_T(t)=\inf\{T_n;T_n>t\}-t.
	\end{equation}
	Then for all fixed $t>0$ and $|\beta|<Q-\alpha$,
	\begin{equation}
	\mathbb{E}\left[e^{\frac{\beta^2}{2}R_T(t)}\right]<\infty.
	\end{equation}
\end{enumerate}
\end{prop}
We provide in Appendix~\ref{app:renewal} a proof for the last claim.

\subsection{Girsanov theorem}
We will apply Girsanov theorem to some exponential functionals of a GFF (or a Brownian motion) in the following form:
\begin{lemm}[Girsanov theorem]\label{lem:Girsanov}
Let $Y$ be some Gaussian variable measurable with respect to a Gaussian free field $X$ and $F$ some bounded functional. Then
\begin{equation}
\mathbb{E}[e^{Y-\frac{\mathbb{E}[Y]^2}{2}}F(X(\cdot))]=\mathbb{E}[F(X(\cdot)+\mathbb{E}[X(\cdot)Y])].
\end{equation}
\end{lemm}

\subsection{Kahane's inequality}
We record two versions of Kahane's Gaussian comparaison inequality.
\begin{lemm}[Kahane's convexity inequality]\label{lem:KahaneConvexity}
Let $X,Y$ be two centered Gaussian field indexed by $T$ such that
\begin{equation}
\mathbb{E}\left[X(i)X(j)\right]\leq\mathbb{E}\left[Y(i)Y(j)\right],\quad \forall (i,j)\in T\times T.
\end{equation}
Then for all non-negative weights $(p_i)_{i\in T}$ and all convex function $F$ with at most polynomial growth at infinity,
\begin{equation}
\mathbb{E}\left[F\left(\sum\limits_{i\in T}p_i e^{X_i-\frac{1}{2}\mathbb{E}[X_i^2]}\right)\right]\leq\mathbb{E}\left[F\left(\sum\limits_{i\in T}p_i e^{Y_i-\frac{1}{2}\mathbb{E}[Y_i^2]}\right)\right].
\end{equation}
\end{lemm}
\begin{proof}
See \cite[Theorem~2.1]{rhodes2014gaussian} for references in English.
\end{proof}
\begin{lemm}[Kahane-Slepian diagonal inequality]\label{lem:KahaneDiagonal}
Let $X,Y$ be two centered Gaussian fields indexed by $T$ such that there exist subsets $A,B\subset T$ on which
\begin{align}
\mathbb{E}\left[X(i)X(j)\right]\leq\mathbb{E}\left[Y(i)Y(j)\right],&\quad \forall(i,j)\in A;\\
\mathbb{E}\left[X(i)X(j)\right]\geq\mathbb{E}\left[Y(i)Y(j)\right],&\quad \forall(i,j)\in B;\\
\mathbb{E}\left[X(i)X(j)\right]=\mathbb{E}\left[Y(i)Y(j)\right],&\quad \forall(i,j)\notin A\cup B.
\end{align}
Suppose $F:\mathbb{R}^{T}\to\mathbb{R}$ is some smooth real functional with appropriate growth at infinity in both its first and second derivatives and such that
\begin{align}
\partial_{ij}F\geq 0,&\quad \forall(i,j)\in A;\\
\partial_{ij}F\leq 0,&\quad \forall(i,j)\in B.
\end{align}
Then we have
\begin{equation}
\mathbb{E}\left[F(X)\right]\leq\mathbb{E}\left[F(Y)\right].
\end{equation}
\end{lemm}
\begin{proof}
See \cite{zeitouniGaussian}, Theorem 3.
\end{proof}

\section{Local study: setup and regularization}\label{sec:LocalSetup}
To prove the Main Theorem~\ref{th:MainTheorem}, it is instumental to study its local version, namely the regularization and convergence near only $1$ insertion point at $z=0$ with complex parameter $\alpha+i\beta$. We will use study two different regularization procedures that yields probabilistically the correct analytic continuation of (local) Liouville correlation functions in the region
\begin{equation}
\mathcal{R}_{loc}=\{\alpha+i\beta\in\mathbb{C}; |\beta|<Q-\alpha\}.
\end{equation}

More precisely, we will seperate two different regimes:
\begin{itemize}
	\item In the region \begin{equation}\mathcal{R}^{M}_{loc}=\mathcal{R}_{loc}\cap\{\alpha>Q-\gamma\}\end{equation} we apply the so-called martingale method, Theorem~\ref{th:LocalMartingale};
	\item In the region \begin{equation}\mathcal{R}^{T}_{loc}=\mathcal{R}_{loc}\cap\{\alpha<Q-\frac{\gamma}{2}\}\end{equation} we apply the so-called stopping time method, Theorem~\ref{th:LocalStoppingTime}.
\end{itemize}

Each of these methods will yield in a probabilistic way a natural analytic continuation of (local) Liouville correlation functions. Since their intersection contains a non empty open set, together they extend the Liouville correlation functions analytically to the whole region $\mathcal{R}_{loc}$.

\subsection{Setup and notations}
We first define the local version of Liouville correlation function which reflects the regularization procedure near one insertion point at $z=0$.
\begin{defi}[Liouville correlation function: local version and regularization]\label{def:LocalCorrelation}
Consider the unit disk $\mathbb{D}\subset\mathbb{C}$ parametrized by $(s,\theta)\in\mathbb{R}_{+}\times[0,2\pi]$ by the following map:
\begin{equation}
(s,\theta)\mapsto e^{-s}e^{i\theta}\in\mathbb{D}.
\end{equation}
Let $X$ be the centered log-correlated Gaussian field on $\mathbb{D}$ of covariance kernel
\begin{equation}
K_\mathbb{D}(x,y)\coloneqq\mathbb{E}[X(x)X(y)]=-\ln|x-y|.
\end{equation}
Then following the radial decomposition of GFF (Lemma~\ref{lem:Radial}) one can decompose $X$ into two independent Gaussian components:
\begin{itemize}
	\item The radial part with can be expressed in terms of a time-changed Brownian motion \begin{equation}B_t=X_{e^{-t}(0)};\end{equation}
	\item The lateral noise part $Y(s,\theta)$ which has covariance kernel \begin{equation}\mathbb{E}[Y(s,\theta)Y(t,\theta')]=\ln\frac{e^{-s}\vee e^{-t}}{|e^{-s}e^{i\theta}-e^{-t}e^{i\theta'}|}.\end{equation} We also use the notation $Z_s$ for the GMC measure associated with the lateral noise $Y$.
\end{itemize}

The local regularized Liouville correlation function with parameter $\alpha+i\beta\in\mathbb{C}$ is defined as
\begin{equation}
G(\alpha+i\beta;T)\coloneqq\mathbb{E}\left[e^{i\beta B_{T}+\frac{\beta^2}{2}T}e^{-\mu\int_{0}^{T}e^{\gamma(B_r-(Q-\alpha)r)}Z_r dr}\right]
\end{equation}
where $T$ denotes some (possibly random) positive time, as long as this expectation can be well defined (e.g. finite).
\end{defi}

\subsection{Convergence and analyticity: martingale method}
In this section, we consider the local regularized Liouville correlation function
\begin{equation}
G(\alpha+i\beta;t)=\mathbb{E}\left[e^{i\beta B_t+\frac{\beta^2}{2}t}e^{-\mu\int_{0}^{t}e^{\gamma(B_r-(Q-\alpha)r)}Z_r dr}\right]
\end{equation}
with fixed deterministic time $t\geq 0$. It is the same regularization as in \cite[Section~4]{kupiainen2017integrability} and it is readily seen that $G$ defines an entire function in $\alpha+i\beta$ for every fixed $t$. We thus study its convergence at $t$ goes to infinity.

As announced before we focus on the region \begin{equation}\mathcal{R}^{M}_{loc}=\mathcal{R}_{loc}\cap\{\alpha>Q-\gamma\}.\end{equation}
\begin{theo}[Local version of the main theorem with fixed time]\label{th:LocalMartingale}
Consider the function
\begin{equation}
G(\alpha+i\beta;t)=\mathbb{E}\left[e^{i\beta B_t+\frac{\beta^2}{2}t}e^{-\mu\int_{0}^{t}e^{\gamma(B_r-(Q-\alpha)r)}Z_r dr}\right].
\end{equation}
We claim that:
\begin{enumerate}
	\item For every $\alpha+i\beta\in\mathbb{C}$, $G(\alpha+i\beta;t)$ is well defined and analytic in $(\alpha,\beta)$ for every finite $t\geq 0$;
	\item For fixed $\alpha+i\beta\in\mathcal{R}^{M}_{loc}$, the limit $G^M(\alpha+i\beta)$ of $G(\alpha+i\beta;t)$ as $t\to\infty$ is well-defined;
	\item The limit $G^M(\alpha+i\beta)$ as a function of $\alpha+i\beta\in\mathbb{C}$ is analytic in $\mathcal{R}^{M}_{loc}$.
\end{enumerate}
\end{theo}
Since for real parameter $\alpha$ this regularization is exactly the original regularization of Liouville correlation function using path integral formalism as defined in \cite{david2016liouville}, this theorem proves that the limit function
\begin{equation}
G^M(\alpha+i\beta)
\end{equation}
is the analytic continuation of the local Liouville correlation function from $\alpha\in\mathbb{R}$ to $\alpha+i\beta\in\mathcal{R}^{M}_{loc}$.

\subsection{Convergence and analyticity: stopping time method}
In this section, we consider the local regularized Liouville correlation function
\begin{equation}
G(\alpha+i\beta;T_N)=\mathbb{E}\left[e^{i\beta B_{T_N}+\frac{\beta^2}{2}T_N}e^{-\mu\int_{0}^{T_N}e^{\gamma(B_r-(Q-\alpha)r)}Z_r dr}\right]
\end{equation}
for $N\in\mathbb{N}$, where $T_N$ is defined as the stopping time for the drifted Brownian motion at level $-N$:
\begin{equation}
T_N=\inf\{t; B_t-(Q-\alpha)t=-N\}.
\end{equation}
We study its convergence as $N$ goes to infinity.

As announced before we focus on the region \begin{equation}\mathcal{R}^{T}_{loc}=\mathcal{R}_{loc}\cap\{\alpha<Q-\frac{\gamma}{2}\}.\end{equation}
\begin{theo}[Local version of the main theorem with stopping time]\label{th:LocalStoppingTime}
Consider the function
\begin{equation}
G(\alpha+i\beta;T_N)=\mathbb{E}\left[e^{i\beta B_{T_N}+\frac{\beta^2}{2}T_N}e^{-\mu\int_{0}^{T_N}e^{\gamma(B_r-(Q-\alpha)r)}Z_r dr}\right].
\end{equation}
We claim that:
\begin{enumerate}
	\item If $|\beta|<Q-\alpha$, then $G(\alpha+i\beta;T_N)$ is well defined for every $N\in\mathbb{N}$;
	\item For fixed $\alpha+i\beta\in\mathcal{R}^{T}_{loc}$, $|G(\alpha+i\beta;T_N)|<C$ for some finite constant $C=C_{\alpha+i\beta}$ independent of $N$ and locally uniform in $\alpha+i\beta$;
	\item For fixed $\alpha+i\beta\in\mathcal{R}^{T}_{loc}$, the limit $G^T(\alpha+i\beta)$ of $G(\alpha+i\beta;T_N)$ as $N\to\infty$ is well-defined;
	\item The limit $G^T(\alpha+i\beta)$ as a function of $\alpha+i\beta\in\mathbb{C}$ is analytic in $\mathcal{R}^{T}_{loc}$.
\end{enumerate}
\end{theo}
It is not readily seen that $G(\alpha+i\beta;T_N)$ is analytic in $\alpha+i\beta\in\mathbb{C}$ for fixed $N$, thus in the proof of the above theorem, it shall suffer from some minor modification which disappears in the limit. Lastly, in the region with non-trivial interior $\mathcal{R}^{M}_{loc}\cap\mathcal{R}^{T}_{loc}$ the two functions $G^{T}(\alpha+i\beta)$ and $G^{M}(\alpha+i\beta)$ coincide (since one is the sequential limit of the other), so $G^{T}(\alpha+i\beta)$ is the correct analytic continuation of the local Liouville correlation function from $\mathcal{R}^{M}_{loc}$ to $\mathcal{R}_{loc}=\mathcal{R}^{M}_{loc}\cup\mathcal{R}^{T}_{loc}$.

The last observation defines the analytic continuation of local Liouville correlation function $G(\alpha)$ for $\alpha<Q$ to $G(\alpha+i\beta)$ with $\alpha+i\beta\in\mathcal{R}_{loc}$, i.e. $\{|\beta|<Q-\alpha\}$, where
\begin{itemize}
	\item $G(\alpha+i\beta)\coloneqq G^{M}(\alpha+i\beta)$ when $\alpha+i\beta\in\mathcal{R}^{M}_{loc}$;
	\item $G(\alpha+i\beta)\coloneqq G^{T}(\alpha+i\beta)$ when $\alpha+i\beta\in\mathcal{R}^{T}_{loc}$.
\end{itemize}

\section{Local study I: martingale method}\label{sec:MartingaleLocal}
We follow notations from Section~\ref{sec:LocalSetup}. In this section we focus on the region
\begin{equation}
\mathcal{R}^{M}_{loc}=\mathcal{R}_{loc}\cap\{\alpha>Q-\gamma\}.
\end{equation}
The main property that we need in this regime is the following:
\begin{prop}[Property: martingale region]\label{prop:MartingaleLocal}
One has the following upper bound on the first derivative of $G(\alpha+i\beta;t)$ with respect to $t$:
\begin{equation}
\left|\frac{\partial G(\alpha+i\beta;t)}{\partial t}\right|\leq Ce^{-\frac{(Q-\alpha)^2}{2}t+\frac{\beta^2}{2}t}
\end{equation}
where $C$ is some finite constant locally uniform in $(\alpha,\beta)$ independent of $t$ when $t\geq 1$.
\end{prop}
\begin{proof}
This is essentially a consequence of the generalized freezing estimate, Appendix~\ref{app:freezing}. It is more convenient to work with the geometry of $\mathbb{D}$, so let us use the following equivalent representation for $G(\alpha+i\beta;t)$ (see Definition~\ref{def:LocalCorrelation} for the conformal change of domain):
\begin{equation}
G(\alpha+i\beta;t)=\mathbb{E}\left[e^{i\beta B_t+\frac{\beta^2}{2}t}e^{-\mu\int_{\mathbb{D}\backslash B(0,e^{-t})}\frac{1}{|x|^{\gamma\alpha}}M_\gamma(d^2x)}\right].
\end{equation}

By Itô calculus on the Brownian motion $B_t$ one can write
\begin{equation}\label{eq:MartingaleIto}
\frac{\partial G(\alpha+i\beta;t)}{\partial t}=\frac{1}{2\pi}\int_{0}^{2\pi}-\mu e^{(\gamma\alpha-2)t}e^{i\gamma\beta t}\mathbb{E}\left[e^{i\beta B_t+\frac{\beta^2}{2}t}e^{-\mu\int_{\mathbb{D}\backslash B(0,e^{-t})}\frac{1}{|x-e^{-t}e^{i\theta}|^{\gamma^2}}\frac{M_\gamma(d^2x)}{|x|^{\gamma\alpha}}}\right]d\theta
\end{equation}
where we recognize a GMC measure with two (real) insertions of respective parameters $\alpha$ and $\gamma$ all contained in $B(0,e^{-t})$. Since by assumption $\alpha+\gamma>Q$, applying Lemma~\ref{lem:GeneralFreezing} yields
\begin{equation}
\left|\frac{\partial G(\alpha+i\beta;t)}{\partial t}\right|\leq\mu Ce^{(\gamma\alpha-2)t}e^{-\frac{(\alpha+\gamma-Q)^2}{2}t}e^{\frac{\beta^2}{2}t}\leq Ce^{-\frac{(Q-\alpha)^2}{2}t+\frac{\beta^2}{2}t}
\end{equation}
with $C$ finite constant locally uniform in $(\alpha,\beta)$, independent of $t$ when $t\geq 1$.
\end{proof}

\subsection{Convergence and analyticity of the limit}
With Proposition~\ref{prop:MartingaleLocal} we are ready to give analytic continuation of local Liouville correlation function in the martinagle region.
\begin{proof}[Proof of Theorem~\ref{th:LocalMartingale}]
It is readily seen in \cite[Section~4]{kupiainen2017integrability} that for every fixed $t$, $G(\alpha+i\beta;t)$ is well defined and analytic in $(\alpha,\beta)$ since it is complex differentiable in $\alpha+i\beta$.

To prove claim (2), notice that if $|\beta|<Q-\alpha$ and $\alpha>Q-\gamma$, then using Proposition~\ref{prop:MartingaleLocal},
\begin{equation}
\left|\frac{\partial G(\alpha+i\beta;t)}{\partial t}\right|
\end{equation}
decays exponentially as $t$ goes to infinity, thus proving the convergence.

To prove claim (3), notice that the exponential convergence rate above is locally uniform in $(\alpha,\beta)$. Since for every $t$, $G(\alpha+i\beta;t)$ is analytic in $(\alpha,\beta)$, local uniform convergence of analytic functions yields analyticity of the limit. As the limit is the same when $\beta=0$, the limit thus defined is indeed the only analytic continuation of local Liouville correlation function from real parameter $\alpha$ to the region
\begin{equation}
\mathcal{R}^{M}_{loc}=\mathcal{R}_{loc}\cap\{\alpha>Q-\gamma\}.
\end{equation}
\end{proof}
\begin{rema}[Extention of martingale method]
One can ask if we can define probabilistically the analytic continuation of local Liouville correlation function to $\mathcal{R}_{loc}$ only using this martingale method. We make several comments on this question: it is possible but the actual proof is much more involved than the current article. In fact, one can look at higher derivatives of the function $G(\alpha+i\beta;t)$ in $\mathcal{R}_{loc}$ but they are in general ill-defined. To renormalize this explosion phenomenon one has to analyse carefully the situation and do a precise Taylor expansion. Getting rid of the first terms in this Taylor expansion appropriately will allow us to push the analytic continuation beyond the region $\mathcal{R}^{M}_{loc}$ and eventually cover the whole region $\mathcal{R}_{loc}$ recursively. However obtaining this Taylor expansion is technically much more involved and will be investigated seperately.
\end{rema}

\section{Local study II: stopping time method}\label{sec:StoppingLocal}
We follow notations from Section~\ref{sec:LocalSetup}. In this section we focus on the region
\begin{equation}
\mathcal{R}^{T}_{loc}=\mathcal{R}_{loc}\cap\{\alpha<Q-\frac{\gamma}{2}\}.
\end{equation}
The main property that we need in this regime is the following:
\begin{prop}[Property: stopping time region]\label{prop:StoppingLocal}
If $\alpha<Q-\frac{\gamma}{2}$, then the positive random measure
\begin{equation}
M_{1}=\int_{0}^{T_1}e^{\gamma(B_r-(Q-\alpha)r)}Z_rdr
\end{equation}
is bounded in $L^1$, i.e.
\begin{equation}
\mathbb{E}\left[\int_{0}^{T_1}e^{\gamma(B_r-(Q-\alpha)r)}Z_rdr\right]<\infty.
\end{equation}
\end{prop}
\begin{proof}
By positivity of the measure we have always the following bound
\begin{equation}
M_{1}\leq\int_{0}^{\infty}e^{\gamma(B_r-(Q-\alpha)r)}Z_rdr
\end{equation}
and when we map the right hand side conformally to the disk, we obtain
\begin{equation}
\int_{B(0,1)}\frac{1}{|z|^{\alpha\gamma}}M_\gamma(dz)
\end{equation}
where $M_\gamma(dz)$ denotes the GMC measure associated to some log-correlated field $X$ in $B(0,2)$.

By the study of generalized Seiberg bound (see \cite[Section~3.4]{david2016liouville}), the last quantity has finite positive moment up to
\begin{equation}
p_c(\alpha)=\frac{4}{\gamma^2}\wedge\frac{2}{\gamma}(Q-\alpha)
\end{equation}
which is strictly bigger than $1$ with our assumption on $\alpha$.
\end{proof}

\subsection{Notations and Brownian motion decorrelation}
Now we proceed towards the study of the function
\begin{equation}
G(\alpha+i\beta;T_N)=\mathbb{E}\left[e^{i\beta B_{T_N}+\frac{\beta^2}{2}T_N}e^{-\mu\int_{0}^{T_N}e^{\gamma(B_r-(Q-\alpha)r)}Z_r dr}\right].
\end{equation}

First we want to use Markov property to cut the time interval $[0,T_N]$ into subintervals with disjoint interiors
\begin{equation}
[0,T_N]=\coprod\limits_{i=0}^{N-1}[T_i,T_{i+1}]
\end{equation}
such that by definition of stopping time for a drifted Brownian motion, we can decorrelate the Brownian motion into independent components on each interval. More precisely, if for each $i$, $B^{i}(s)$ denotes the fluctuation of the Brownian motion $B$ from time $T_i$ until $T_{i+1}$, i.e.
\begin{equation}
B^{i}(s)\coloneqq B(s+T_i)-B(T_i),\quad s\leq T_{i+1}-T_i,
\end{equation}
then if $i\neq j$, the Brownian paths $B^{i}$ and $B^{j}$ are mutually independent.

By Markov property we consequently write $G(\alpha+i\beta;T_N)$ as
\begin{equation}
G(\alpha+i\beta;T_N)=\mathbb{E}\left[\prod_{i=0}^{N-1}\left(e^{i\beta (B_{T_{i+1}}-B_{T_{i}})+\frac{\beta^2}{2}(T_{i+1}-T_{i})}e^{-\mu e^{-\gamma i}\int_{0}^{T_{i+1}-T_{i}}e^{\gamma(B^i_r-(Q-\alpha)r)}Z_{r+T_{i}}dr}\right)\right].
\end{equation}

If we let $M_{i}$ denote the positive random measure
\begin{equation}
M_{i}\coloneqq \int_{0}^{T^{i}_1}e^{\gamma(B^i_r-(Q-\alpha)r)}Z_{r+T_{i}}dr
\end{equation}
with $T^{i}_{1}=T_{i+1}-T_i$, the stopping time associated to the drifted Brownain motion $B^{i}_r-(Q-\alpha)r$ at level $-1$, then the above expression can be condensed into
\begin{equation}
G(\alpha+i\beta;T_N)=\mathbb{E}\left[\prod_{i=0}^{N-1}\left(e^{i\beta (B_{T_{i+1}}-B_{T_{i}})+\frac{\beta^2}{2}(T_{i+1}-T_{i})}e^{-\mu e^{-\gamma i}M_{i}}\right)\right].
\end{equation}
Notice that for every $i$, $M_{i}$ is distributed as $M_{1}$ in Proposition~\ref{prop:StoppingLocal}, but they are correlated (only in the lateral noise part $Z_r$). The correlation between $M_{i},M_{j}$ decays as the corresponding time intervals $[T_{i},T_{i+1}],[T_{j},T_{j+1}]$ separate apart.

\subsection{Preliminary observation}
The first observation we will make is that terms where the measure $M_{i}$ is not too big can be bounded uniformly by some global constant $C$ which is locally uniform in $(\alpha,\beta)$.

More precisely, let $N$ be arbitrary and $J\subset[|0,N-1|]$ any subset of the set
\begin{equation}
[|0,N-1|]\coloneqq \{0,1,\dots,N-1\}
\end{equation}
and consider the function, where the measure $M_i$ is truncated at level $e^{\frac{\gamma i}{2}}$,
\begin{equation}
g^{J}(\alpha+i\beta;T_N)=\mathbb{E}\left[\prod_{i\in J}\left(e^{i\beta (B_{T_{i+1}}-B_{T_{i}})+\frac{\beta^2}{2}(T_{i+1}-T_{i})}e^{-\mu e^{-\gamma i}(M_{i}\wedge e^{\frac{\gamma i}{2}})}\right)\right].
\end{equation}
We claim that
\begin{prop}\label{prop:SmallIndicator}
With the above notation, there exists some constant $C_0$ independent of $N$ and $J$ and locally uniform in $(\alpha,\beta)$ such that the following bound holds:
\begin{equation}
\left|g^{J}(\alpha+i\beta;T_N)\right|\leq C_0.
\end{equation}
Nonetheless the bound $C_0$ depends on $\mu$ and we will give an estimate in a corollary.
\end{prop}
\begin{proof}
Write
\begin{equation}
e^{-\mu e^{-\gamma i}(M_{i}\wedge e^{\frac{\gamma i}{2}})}=1-\left(1-e^{-\mu e^{-\gamma i}(M_{i}\wedge e^{\frac{\gamma i}{2}})}\right)
\end{equation}
and expand all the product over subsets $K\subset J$. By Markov property, for any subset $K\subset J$ the corresponding term in the product expansion
\begin{equation}
\mathbb{E}\left[\prod_{i\in J}\left(e^{i\beta (B_{T_{i+1}}-B_{T_{i}})+\frac{\beta^2}{2}(T_{i+1}-T_{i})}\right)\prod_{i\in K}\left(1-e^{-\mu e^{-\gamma i}(M_{i}\wedge e^{\frac{\gamma i}{2}})}\right)\right]
\end{equation}
can be bounded in absolute value by (with $C$ independent of $\mu$)
\begin{equation}
\mathbb{E}\left[\prod_{i\in K}\left(e^{\frac{\beta^2}{2}(T_{i+1}-T_{i})}\left(1-e^{-\mu e^{-\frac{\gamma i}{2}}}\right)\right)\right]=C^{\#(K)}\prod_{i\in K}\left(1-e^{-\mu e^{-\frac{\gamma i}{2}}}\right)
\end{equation}
and summing up we get an upper bound for $g^{J}(\alpha+i\beta;T_N)$ in absolute value:
\begin{equation}\label{eq:SmallIndicatorUpperBound}
\prod_{i\in J}\left(1+C\left(1-e^{-\mu e^{-\frac{\gamma i}{2}}}\right)\right)\leq\prod_{i\in\mathbb{N}}\left(1+C\left(1-e^{-\mu e^{-\frac{\gamma i}{2}}}\right)\right).
\end{equation}
To see that this is bounded independent of $I$, using $1-e^{-x}\leq x$ for $x\geq 0$ and write
\begin{equation}
\prod_{i\in\mathbb{N}}\left(1+C\left(1-e^{-\mu e^{-\frac{\gamma i}{2}}}\right)\right)\leq\prod_{i\in\mathbb{N}}\left(1+C\mu e^{-\frac{\gamma i}{2}}\right)<\infty.
\end{equation}
This finishes the proof.
\end{proof}
We study the upper bound when $\mu$ varies and record the following corollary:
\begin{coro}\label{cor:SmallIndicatorMu}
Let $c>0$ and write $\mu=\mu_0 e^{\gamma c}$ for $\mu_0>0$. The bound $C_0(\mu)$ in Proposition~\ref{prop:SmallIndicator} can be bounded above by
\begin{equation}
(1+C)^{\frac{c+1}{2}} C_0(\mu_0)
\end{equation}
\end{coro}
\begin{proof}
Let us restart from the Equation~\eqref{eq:SmallIndicatorUpperBound} and use the slighly better bound
\begin{equation}
1-e^{-x}\leq x\wedge 1
\end{equation}
for $x\geq 0$. Now if $2c=k\in\mathbb{N}$ then
\begin{equation}
\prod_{i\in\mathbb{N}}\left(1+C\left(1-e^{-\mu e^{-\frac{\gamma i}{2}}}\right)\right)=\prod_{i=0}^{k-1}\left(1+C\left(1-e^{-\mu e^{-\frac{\gamma i}{2}}}\right)\right)\prod_{i\in\mathbb{N}}\left(1+C\left(1-e^{-\mu_0 e^{-\frac{\gamma i}{2}}}\right)\right)
\end{equation}
where we recognize the upper bound
\begin{equation}
(1+C)^{k} C_0(\mu_0).
\end{equation}
For general $c$, a similar argument yields the upper bound
\begin{equation}
(1+C)^{\frac{c+1}{2}} C_0(\mu_0)
\end{equation}
as desired.
\end{proof}

Consequently, we will decompose the function $G$ with respect to the subset $J\in[|0,N-1|]$ where the measures $M_i$ are extremely large. Indeed, we can write
\begin{equation}
G(\alpha+i\beta;T_N)=\sum\limits_{J\subset[|0,N-1|]}G^{J}(\alpha+i\beta;T_N)
\end{equation}
where $G(J)$ denotes the configuration where measures with indices in $J$ are extremely large:
\begin{align}
G^{J}(\alpha+i\beta;T_N)\coloneqq\mathbb{E}\bigg[&\prod_{i\in J}\left(e^{i\beta (B_{T_{i+1}}-B_{T_{i}})+\frac{\beta^2}{2}(T_{i+1}-T_{i})}\left(e^{-\mu e^{-\gamma i}M_{i}}-e^{-\mu e^{-\frac{\gamma i}{2}}}\right)\mathbf{1}_{\{M_{i}>e^{\frac{\gamma i}{2}}\}}\right)\\
\times&\prod_{i\in J^c}\left(e^{i\beta (B_{T_{i+1}}-B_{T_{i}})+\frac{\beta^2}{2}(T_{i+1}-T_{i})}e^{-\mu e^{-\gamma i}(M_{i}\wedge e^{\frac{\gamma i}{2}})}\right)\bigg]\nonumber
\end{align}
with abused notation $J^c$ denoting the complement of $J$ in $[|0,N-1|]$.

Finding a proper way to translate the large deviation phenomenon
\begin{equation}
\prod_{i\in J}\mathbf{1}_{\{M_{i}>e^{\frac{\gamma i}{2}}\}}
\end{equation}
for multiple indices simulteneously will be the focus of the following sections.

\subsection{Cutting annuli and Kahane's inequality}
One crucial step in the proof using stopping time method is to decorrelate the random measures $\{M_i\}_{i\in\mathbb{N}}$ in a suitable way. The classical Gaussian comparaison inequalities allow one to do so to some extend, usually under some weak correlation assumption. We can obtain such assumption if we force the time intervals of $M_i$ to be far away mutually. Indeed,
\begin{prop}\label{prop:WeakCorrelation}
Let $\epsilon$ be arbitraty positive constant. There exists some constant $\delta$ only depending on $\epsilon$ such that if $|t-s|>\delta$, for all $(\theta,\theta')\in[0,2\pi]^2$,
\begin{equation}
\mathbb{E}[Y(s,\theta)Y(t,\theta')]<\epsilon.
\end{equation}
\end{prop}
\begin{proof}
Recall the lateral noise correlation function
\begin{equation}
\mathbb{E}[Y(s,\theta)Y(t,\theta')]=\ln\frac{e^{-s}\vee e^{-t}}{|e^{-s}e^{i\theta}-e^{-t}e^{i\theta'}|}.
\end{equation}
Without loss of generosity, let us suppose $s<t$, and the above expression equals
\begin{equation}
\mathbb{E}[Y(t,\theta)Y(s,\theta')]=\ln\frac{1}{|1-e^{-(t-s)e^{i(\theta-\theta')}}|}\leq\ln\frac{1}{1-e^{-(t-s)}}.
\end{equation}
In particular, any $\delta$ such that
\begin{equation}
\delta>\ln\frac{1}{1-e^{-\epsilon}}
\end{equation}
will satisfy the claim.
\end{proof}

To apply this observation to our case, we have to specify a certain kind of subset of $[|0,N-1|]$ (measurable with respect to the Brownian motion $B$) where the above weak correlation assumption is met. We will then apply Kahane's inequality to decorrelate the random measures on this subset.
\begin{defi}[Cutting annuli]\label{def:CuttingAnnuli}
Let $B$ be the Brownian motion corresponding to the radial part of the GFF $X$ as defined above. We call a subset $I=\{i_1<i_2<\dots<i_k\}\subset[|0,N-1|]$ a $(B,\delta)$-cutting annuli if the following condition is met:
\begin{equation}
\forall i_j\in I,\quad |T_{i_{j+1}}-T_{i_j+1}|>\delta.
\end{equation}
Otherwise put, for all different indices $i\neq j$, the intervals $[i_i,i_{i+1}]$ and $[i_j,i_{j+1}]$ are seperated by distance $\delta$ if $(i_i,i_j)\in I^2$.
\end{defi}

It is useful to reformulate the above definition in another way. For any indice $i\in\mathbb{N}$, define the $(B,\delta)$-precutting image of $i$, denoted by $\overline{i^{(B,\delta)}}$, in the following way:
\begin{defi}[Precutting image]\label{def:PreCuttingImage}
Let $i\in\mathbb{N}$. We define the $(B,\delta)$-precutting image of $i$ to be the only indice $\overline{i^{(B,\delta)}}$ such that
\begin{equation}
T_{i}-1\in[T_{\overline{i^{(B,\delta)}}}, T_{\overline{i^{(B,\delta)}}+1}).
\end{equation}
If such indice does not exist, then we define $\overline{i^{(B,\delta)}}$ to be $-1$. We drop the indices on $(B,\delta)$ and simply write $\overline{i}$ if there is no ambiguity.
\end{defi}

Then for every ordered set $I=\{i_1<i_2<\dots\}\subset\mathbb{N}$, $I$ is a $(B,\delta)$-cutting annuli if and only if
\begin{equation}
i_1<\overline{i_2}<i_2<\overline{i_3}<\dots
\end{equation}
is satisfied. If this is the case, we define the an extended version of $I$ in the following manner:
\begin{defi}[Extended cutting annuli]\label{def:ExtendedCuttingAnnuli}
Let $I=\{i_1<i_2<\dots\}\subset\mathbb{N}$ be an ordered set. We will use the notation
\begin{equation}
I\leftrightarrow(B,\delta)
\end{equation}
to denote that $I$ is a $(B,\delta)$-cutting annuli. If it is the case, define
\begin{equation}
\overline{I^{(B,\delta)}}\coloneqq\{\overline{i_1}<i_1<\overline{i_2}<i_2<\overline{i_3}<\dots\}
\end{equation}
to be the $(B,\delta)$-extended version of $I$. We drop the indice $(B,\delta)$ when there is no ambiguity.
\end{defi}

Now we are able to state the Gaussian decorrelation inequality that we will be depending on.
\begin{lemm}[Kahane's decorrelation]\label{lem:KahaneDecorrelation}
Let
\begin{equation}
(\alpha,\beta)\in\mathcal{R}^{T}_{loc}=\mathcal{R}_{loc}\cap\{\alpha<Q-\frac{\gamma}{2}\}.
\end{equation}
Consider the function
\begin{equation}
\overline{G^{I}}(\alpha+i\beta;T_N)\coloneqq\mathbb{E}\left[\mathbf{1}_{I\leftrightarrow(B,\delta)}\prod_{i\in I}\left(e^{\frac{\beta^2}{2}(T_{i+1}-T_{i})}\mathbf{1}_{\{M_{i}>e^{\frac{\gamma i}{2}}\}}\right)\right].
\end{equation}
Then when $\delta>\ln\frac{1}{1-e^{-\frac{1}{8\gamma}}}$, we have the following bound for any $q>\frac{(Q-\alpha)^2}{(Q-\alpha)^2-\beta^2}$:
\begin{equation}
|\overline{G^{I}}(\alpha+i\beta;T_N)|\leq C_0\times C^{\#(I)}\prod\limits_{i\in I}e^{-\frac{\gamma i}{4q}}
\end{equation}
where $C_0, C$ are constants independent of $N$ and $I$ and locally uniform in $(\alpha,\beta)$.
\end{lemm}
\begin{proof}
If 
\begin{equation}
\mathbf{1}_{I\leftrightarrow(B,\delta)}=0
\end{equation}
then the claim is trivial. We thus assume that
\begin{equation}
I\leftrightarrow(B,\delta)
\end{equation}
holds. With this assumption in mind, if $Y_i(t,\theta)$ denotes the underlying log-correlated fields of the measures $M_{i}$, i.e. if
\begin{equation}
M_{i}=\int_{0}^{T^{i}_1}e^{\gamma(B^{i}_r-(Q-\alpha)r)}\left(\frac{1}{2\pi}\int_{0}^{2\pi}e^{\gamma Y_i(r,\theta)}d\theta\right)dr
\end{equation}
then by Proposition~\ref{prop:WeakCorrelation}, if $i,j\in I$ and $i\neq j$ then for all $s,t\in[0,T^{i}_1]\times[0,T^{j}_1]$,
\begin{equation}
\mathbb{E}[Y_i(s,\theta)Y_j(t,\theta')]<\epsilon
\end{equation}
where $\epsilon>0$ is such that
\begin{equation}
e^{-\delta}+e^{-\epsilon}=1.
\end{equation}

We want to use this information and apply the diagonal decorrelation inequality of Kahane-Slepian, Lemma~\ref{lem:KahaneDiagonal}. Consider the following modified GMC measures with a collection of i.i.d. standard Gaussian variables denoted by $\widetilde{N},(N_i)_{i\in\mathbb{N}}$.:
\begin{equation}
\overline{M_i}=\int_{0}^{T^{i}_1}e^{\gamma(B^{i}_r-(Q-\alpha)r)}\left(\frac{1}{2\pi}\int_{0}^{2\pi}e^{\gamma (Y_i(r,\theta)+\sqrt{\epsilon}N_i)}d\theta\right)dr=e^{\gamma\sqrt{\epsilon}N_i}M_i
\end{equation}
\begin{equation}
\widehat{M_i}=\int_{0}^{T^{i}_1}e^{\gamma(B^{i}_r-(Q-\alpha)r)}\left(\frac{1}{2\pi}\int_{0}^{2\pi}e^{\gamma (\widetilde{Y_i}(r,\theta)+\sqrt{\epsilon}\widetilde{N})}d\theta\right)dr=e^{\gamma\sqrt{\epsilon}\widetilde{N}}\widetilde{M_i}
\end{equation}
where $(\widetilde{Y_i})_{i\in\mathbb{N}}$ are mutually independent copies of $(Y_i)_{i\in\mathbb{N}}$ (similarly for $\widetilde{M_i}$). Since for all $(i,j)$,
\begin{equation}
\mathbb{E}[(Y_i(s,\theta)+\sqrt{\epsilon}N_i)(Y_j(t,\theta')+\sqrt{\epsilon}N_j)]\leq\mathbb{E}[(\widetilde{Y_i}(s,\theta)+\sqrt{\epsilon}\widetilde{N})(\widetilde{Y_j}(t,\theta')+\sqrt{\epsilon}\widetilde{N})]
\end{equation}
and $Y+\sqrt{\epsilon}N,\widetilde{Y}+\sqrt{\epsilon}\widetilde{N}$ have the same variance, Kahane-Slepian inequality Lemma~\ref{lem:KahaneDiagonal} yields, for every set $I\leftrightarrow(B,\delta)$
\begin{equation}
\mathbb{E}\left[\mathbf{1}_{I\leftrightarrow(B,\delta)}\prod_{i\in I}\overline{M_i}\right]\leq\mathbb{E}\left[\mathbf{1}_{I\leftrightarrow(B,\delta)}\prod_{i\in I}\widehat{M_i}\right]
\end{equation}
which is equivalent to
\begin{equation}
\mathbb{E}\left[\prod_{i\in I}e^{\gamma\sqrt{\epsilon}N_i}\right]\mathbb{E}\left[\mathbf{1}_{I\leftrightarrow(B,\delta)}\prod_{i\in I}M_i\right]\leq\mathbb{E}\left[\prod_{i\in I}e^{\gamma\sqrt{\epsilon}\widetilde{N}}\right]\mathbb{E}\left[\mathbf{1}_{I\leftrightarrow(B,\delta)}\prod_{i\in I}\widetilde{M_i}\right].
\end{equation}
We retain the weaker relation
\begin{equation}
\mathbb{E}\left[\mathbf{1}_{I\leftrightarrow(B,\delta)}\prod_{i\in I}M_i\right]\leq e^{\#(I)^2\frac{\gamma^2}{2}\epsilon}\prod_{i\in I}\mathbb{E}\left[M_i\right].
\end{equation}

To finish the prove we use Markov inequality on the indicator and the property in the stopping time region from Proposition~\ref{prop:StoppingLocal}: remember
\begin{equation}\label{eq:StoppingTimeFirstBound}
\left|\overline{G^{I}}(\alpha+i\beta;T_N)\right|=\mathbb{E}\left[\mathbf{1}_{I\leftrightarrow(B,\delta)}\prod_{i\in I}\left(e^{\frac{\beta^2}{2}(T_{i+1}-T_{i})}\mathbf{1}_{\{M_{i}>e^{\frac{\gamma i}{2}}\}}\right)\right]
\end{equation}
and use Hölder inequality with a pair $\frac{1}{p}+\frac{1}{q}=1$, $0<p<\frac{(Q-\alpha)^2}{\beta^2}$, i.e.
\begin{equation}
q>\frac{(Q-\alpha)^2}{(Q-\alpha)^2-\beta^2}
\end{equation}
to have
\begin{equation}
\mathbb{E}\left[\mathbf{1}_{I\leftrightarrow(B,\delta)}\prod_{i\in I}\left(e^{\frac{\beta^2}{2}(T_{i+1}-T_{i})}\mathbf{1}_{\{M_{i}>e^{\frac{\gamma i}{2}}\}}\right)\right]\leq\mathbb{E}\left[\prod_{i\in I}e^{\frac{p\beta^2}{2}(T_{i+1}-T_{i})}\right]^{\frac{1}{p}}\mathbb{E}\left[\mathbf{1}_{I\leftrightarrow(B,\delta)}\prod_{i\in I}\mathbf{1}_{\{M_{i}>e^{\frac{\gamma i}{2}}\}}\right]^{\frac{1}{q}}
\end{equation}
By the assumption on $p$, for each $i\in I$ the expectation is bounded uniformly
\begin{equation}
\mathbb{E}\left[e^{\frac{p\beta^2}{2}(T_{i+1}-T_{i})}\right]<C
\end{equation}
with some constant $C$ locally uniform in $(\alpha,\beta)$. Next, use
\begin{equation}
\mathbf{1}_{\{M_{i}>e^{\frac{\gamma i}{2}}\}}\leq e^{-\frac{\gamma i}{2}}M_i
\end{equation}
and the above decorrelation inequality yields
\begin{equation}\label{eq:StoppingTimeHolder}
\mathbb{E}\left[\mathbf{1}_{I\leftrightarrow(B,\delta)}\prod_{i\in I}\mathbf{1}_{\{M_{i}>e^{\frac{\gamma i}{2}}\}}\right]^{\frac{1}{q}}\leq\prod_{i\in I}e^{-\frac{\gamma i}{2}}\mathbb{E}\left[\mathbf{1}_{I\leftrightarrow(B,\delta)}\prod_{i\in I}M_i\right]^{\frac{1}{q}}\leq e^{\#(I)^2\frac{\gamma^2}{2}\epsilon}\prod_{i\in I}e^{-\frac{\gamma i}{2}}\prod_{i\in I}\mathbb{E}\left[M_i\right].
\end{equation}

Now by the assumption on the stopping time region Proposition~\ref{prop:StoppingLocal}, we can bound
\begin{equation}
\prod_{i\in I}\mathbb{E}\left[M_i\right]<C^{\#(I)}
\end{equation}
with some constant $C$ locally uniform in $(\alpha,\beta)$. Also, for
\begin{equation}
\epsilon<\frac{1}{8\gamma}\quad \text{or equivalently}\quad\delta>\ln\frac{1}{1-e^{-\frac{1}{8\gamma}}}
\end{equation}
we have (for $\#(I)\geq 2$, the case $\#(I)<2$ can be treated similarly)
\begin{equation}
e^{\#(I)^2\frac{\gamma^2}{2}\epsilon}\prod_{i\in I}e^{-\frac{\gamma i}{2}}\leq \prod_{i\in I}e^{-\frac{\gamma i}{4}}
\end{equation}
since the indices in $I$ are positive and distinct.

Finally Equations~\eqref{eq:StoppingTimeFirstBound} and \eqref{eq:StoppingTimeHolder} yield, with any $\delta>\ln\frac{1}{1-e^{-\frac{1}{8\gamma}}}$ and $q>\frac{(Q-\alpha)^2}{(Q-\alpha)^2-\beta^2}$,
\begin{equation}
\left|\overline{G^{I}}(\alpha+i\beta;T_N)\right|\leq C^{\#(I)}\prod_{i\in I}e^{-\frac{\gamma i}{4q}}
\end{equation}
with $C$ locally uniform in $(\alpha,\beta)$: this is the claim.
\end{proof}

\begin{rema}
Notice that this upper bound does not depend on $\mu$, since this parameter does not appear in the definition of $\overline{G^{I}}(\alpha+i\beta;T_N)$.
\end{rema}

\subsection{Renewal theory and boundedness of local Liouville correlation function}
In this subsection we will prove another estimation:
\begin{lemm}\label{lem:RenewalBound}
We use the same notations as above and let
\begin{equation}
(\alpha,\beta)\in\mathcal{R}^{T}_{loc}=\mathcal{R}_{loc}\cap\{\alpha<Q-\frac{\gamma}{2}\}.
\end{equation}
We claim that for any fixed $\delta$,
\begin{equation}
|G(\alpha+i\beta;T_N)|\leq C_0\times\sum\limits_{J\in[|0,N-1|]}C^{\#(J)}\left|\overline{G^{J}}(\alpha+i\beta;T_N)\right|
\end{equation}
where $C_0, C$ are constants independent of $N$ and $J$ and locally uniform in $(\alpha,\beta)$.
\end{lemm}

As a corollary of the two previous lemmas we get the following consequence:
\begin{lemm}[Boundedness of local Liouville correlation function]\label{lem:Boundedness}
Let
\begin{equation}
(\alpha,\beta)\in\mathcal{R}^{T}_{loc}=\mathcal{R}_{loc}\cap\{\alpha<Q-\frac{\gamma}{2}\}.
\end{equation}
There exists some constant $C$ independent of $N$ and locally uniform in $(\alpha,\beta)$ such that
\begin{equation}
|G(\alpha+i\beta;T_N)|<C.
\end{equation}
\end{lemm}
\begin{proof}
Combining the two previous lemmas, for any $\delta>\ln\frac{1}{1-e^{-\frac{1}{8\gamma}}}$ and $q>\frac{(Q-\alpha)^2}{(Q-\alpha)^2-\beta^2}$,
\begin{equation}
|G(\alpha+i\beta;T_N)|\leq C_0\times\sum\limits_{J\in[|0,N-1|]}C^{\#(J)}\prod_{j\in J}e^{-\frac{j\gamma}{4q}}\leq C_0\times\prod\limits_{j\in\mathbb{N}}(1+Ce^{-\frac{j\gamma}{4q}})<C
\end{equation}
and one checks that all constants are locally uniform in $(\alpha,\beta)$.
\end{proof}

The rest of this subsection will be devoted to the proof of Lemma~\ref{lem:RenewalBound}.
\begin{proof}[Proof of Lemma~\ref{lem:RenewalBound}]
The proof depends heavily on the definition of cutting annuli.

We first define a map $\Phi_{(B,\delta)}$, depending on $(B,\delta)$, from the subsets of $\mathbb{N}$ to the sets of cutting annuli.
\begin{defi}[Reduction map to cutting annuli]
Let $J=\{j_1<\dots<j_k\}$ be any finite ordered subset of $\mathbb{N}$. The reduction map $\Phi_{(B,\delta)}$ maps $J$ to another ordered subset $I=\{i_1<\dots<i_l\}$ of $\mathbb{N}$ in the following way (from the largest element downward):
\begin{itemize}
	\item $i_l=j_k$;
	\item For $j=l-1$, $i_j=\sup\{a\in J; T_{a+1}+\delta<T_{i_{j+1}}=T_{i_l}\}$.
	\item Recursively so for $j=l-2$, $j=l-3$ until the algorithm stops.
\end{itemize}
Then $l$ is defined as the number of elements in $I$. Observe the two following properties of this map:
\begin{enumerate}
	\item $I$ is a subset of $J$;
	\item $I$ is a $(B,\delta)$-cutting annuli.
\end{enumerate}
\end{defi}
Intuitively speaking, this is a greedy way of choosing elements from $J$ downward in such a way that the time intervals they represent are seperated apart by distance $\delta$. Although $I$ depends on $(B,\delta)$, we don't keep that in the notation when there is no ambiguity.

We will use an important observation on the preimage by $\Phi_{(B,\delta)}$ of a cutting annuli. Recall the Definition~\ref{def:ExtendedCuttingAnnuli} that couples any $(B,\delta)$-cutting annuli $I$ with an extended version $\overline{I}$.
\begin{prop}[Preimage by reduction map]\label{prop:ReductionPreimage}
Let $I$ be a $(B,\delta)$-cutting annuli an $\overline{I}$ the $(B,\delta)$-extended version of $I$. Suppose that $i_0=-\infty$ and
\begin{equation}
\overline{I}=\{\overline{i_1}<i_1<\overline{i_2}<i_2<\overline{i_3}<\dots<\overline{i_l}<i_l\}.
\end{equation}
Then $J\in\Phi_{(B,\delta)}^{-1}(I)$ if and only if:
\begin{enumerate}
	\item $I\subset J$, i.e. for all $i_j\in I$, $i_j\in J$.
	\item Let $i_k\in I$, if $i_{k}<j<\overline{i_{k+1}}$, then $j\neq J$.
\end{enumerate}
For $\overline{i_{k}}\leq j<i_{k}$, $j$ can be in $J$ or not: this gives exactly the cardinal of $\Phi_{(B,\delta)}^{-1}(I)$, although we don't need that
number later.
\end{prop}
We see from this observation that $\Phi_{(B,\delta)}$ is surjective, say from the set of subsets of $[|0,N-1|]$ to the set of cutting annuli of $[|0,N-1|]$: this is because for any cutting annuli $I$, $I$ itself is contained in the preimage $\Phi_{(B,\delta)}^{-1}(I)$.

Now let us return to the study of the function $G$. Remember
\begin{equation}
G(\alpha+i\beta;T_N)=\sum\limits_{J\subset[|0,N-1|]}G^{J}(\alpha+i\beta;T_N)
\end{equation}
with $G(J)$ denotes the configuration where measures with indices in $J$ are extremely large
\begin{align}
G^{J}(\alpha+i\beta;T_N)\coloneqq\mathbb{E}\bigg[&\prod_{i\in J}\left(e^{i\beta (B_{T_{i+1}}-B_{T_{i}})+\frac{\beta^2}{2}(T_{i+1}-T_{i})}\left(e^{-\mu e^{-\gamma i}M_{i}}-e^{-\mu e^{-\frac{\gamma i}{2}}}\right)\mathbf{1}_{\{M_{i}>e^{\frac{\gamma i}{2}}\}}\right)\\
\times&\prod_{i\in J^c}\left(e^{i\beta (B_{T_{i+1}}-B_{T_{i}})+\frac{\beta^2}{2}(T_{i+1}-T_{i})}e^{-\mu e^{-\gamma i}(M_{i}\wedge e^{\frac{\gamma i}{2}})}\right)\bigg]\nonumber
\end{align}
Now rewrite the sum in the following way, depending on the cutting annuli $I=\Phi_{(B,\delta)}(J)$:
\begin{equation}\label{eq:ReductionSumAnnuli}
G(\alpha+i\beta;T_N)=\sum\limits_{I\subset[|0,N-1|]}\mathbb{E}\left[\mathbf{1}_{I\leftrightarrow(B,\delta)}\sum\limits_{J\in\Phi_{(B,\delta)}^{-1}(I)}h_J\right]
\end{equation}
where the notation $h_J$ stands for
{\small
\begin{equation}
\prod_{i\in J}\left(e^{i\beta (B_{T_{i+1}}-B_{T_{i}})+\frac{\beta^2}{2}(T_{i+1}-T_{i})}\left(e^{-\mu e^{-\gamma i}M_{i}}-e^{-\mu e^{-\frac{\gamma i}{2}}}\right)\mathbf{1}_{\{M_{i}>e^{\frac{\gamma i}{2}}\}}\right)\prod_{i\in J^c}\left(e^{i\beta (B_{T_{i+1}}-B_{T_{i}})+\frac{\beta^2}{2}(T_{i+1}-T_{i})}e^{-\mu e^{-\gamma i}(M_{i}\wedge e^{\frac{\gamma i}{2}})}\right).
\end{equation}
}

Let us fix an ordered cutting annuli $I=\{i_1<i_2<\dots<i_l\}\subset[|0,N-1|]$ and focus on controlling
\begin{equation}
\left|\mathbb{E}\left[\mathbf{1}_{I\leftrightarrow(B,\delta)}\sum\limits_{J\in\Phi_{(B,\delta)}^{-1}(I)}h_J\right]\right|.
\end{equation}
By Proposition~\ref{prop:ReductionPreimage} on the preimage $\Phi_{(B,\delta)}^{-1}(I)$ above, this quantity is equal to
\begin{align}
\bigg|\mathbb{E}\bigg[\mathbf{1}_{I\leftrightarrow(B,\delta)}
&\prod_{k=1}^{l}\left(e^{i\beta (B_{T_{i_k+1}}-B_{T_{i_k}})+\frac{\beta^2}{2}(T_{i_k+1}-T_{i_k})}\left(e^{-\mu e^{-\gamma i_k}M_{i_k}}-e^{-\mu e^{-\frac{\gamma i_k}{2}}}\right)\mathbf{1}_{\{M_{i_k}>e^{\frac{\gamma i_k}{2}}\}}\right)\\
&\prod_{k=1}^{l}\prod_{i_{k-1}<j<\overline{i_k}}\left(e^{i\beta (B_{T_{j+1}}-B_{T_{j}})+\frac{\beta^2}{2}(T_{j+1}-T_{j})}e^{-\mu e^{-\gamma j}(M_{j}\wedge e^{\frac{\gamma j}{2}})}\right)\nonumber\\
&\prod_{k=1}^{l}\prod_{\overline{i_k}\leq j<i_k}\left(e^{i\beta (B_{T_{j+1}}-B_{T_{j}})+\frac{\beta^2}{2}(T_{j+1}-T_{j})}e^{-\mu e^{-\gamma j}M_{j}}\right)\bigg]\bigg|\nonumber
\end{align}
and with the argument of Proposition~\ref{prop:SmallIndicator} applied to the products of terms with small indicators, we can majorize this by
\begin{equation}
C_0\mathbb{E}\left[\mathbf{1}_{I\leftrightarrow(B,\delta)}\prod_{k=1}^{l}\left(e^{\frac{\beta^2}{2}(T_{i_k+1}-T_{i_k})}\mathbf{1}_{\{M_{i_k}>e^{\frac{\gamma i_k}{2}}\}}\right)\prod_{k=1}^{l}\prod_{\overline{i_k}\leq j<i_k}\left(e^{\frac{\beta^2}{2}(T_{j+1}-T_{j})}\right)\right]
\end{equation}
with $C_0$ locally uniform in $(\alpha,\beta)$, which is
\begin{equation}
C_0\mathbb{E}\left[\mathbf{1}_{I\leftrightarrow(B,\delta)}\prod_{k=1}^{l}\left(e^{\frac{\beta^2}{2}(T_{i_k+1}-T_{i_k})}\mathbf{1}_{\{M_{i_k}>e^{\frac{\gamma i_k}{2}}\}}\right)\prod_{k=1}^{l}\left(e^{\frac{\beta^2}{2}(T_{i_k}-T_{\overline{i_k}})}\right)\right].
\end{equation}

Now we are not far from being able to apply Kahane's decorrelation Lemma~\ref{lem:KahaneDecorrelation}: it suffices to get rid of the last product term in the above expression. By a standard argument in renewal theory (see Appendix~\ref{app:renewal}), we know that
\begin{equation}
\mathbb{E}\left[e^{\frac{\beta^2}{2}(T_{i_k}-T_{\overline{i_k}})}\right]<C
\end{equation}
for some $C$ locally uniform in $(\alpha,\beta)$, so that by Markov property of the Brownian motion one arrives at the following bound
\begin{equation}
\left|\mathbb{E}\left[\mathbf{1}_{I\leftrightarrow(B,\delta)}\sum\limits_{J\in\Phi_{(B,\delta)}^{-1}(I)}h_J\right]\right|\leq C_0\times C^{\#(I)}\mathbb{E}\left[\mathbf{1}_{I\leftrightarrow(B,\delta)}\prod_{k=1}^{l}\left(e^{\frac{\beta^2}{2}(T_{i_k+1}-T_{i_k})}\mathbf{1}_{\{M_{i_k}>e^{\frac{\gamma i_k}{2}}\}}\right)\right]
\end{equation}
which is with the previous notations
\begin{equation}
\left|\mathbb{E}\left[\mathbf{1}_{I\leftrightarrow(B,\delta)}\sum\limits_{J\in\Phi_{(B,\delta)}^{-1}(I)}h_J\right]\right|\leq C_0\times C^{\#(I)}\left|\overline{G^{I}}(\alpha+i\beta;T_N)\right|.
\end{equation}
Summing up over all possibilities $I\subset[|0,N-1|]$ by Equation~\eqref{eq:ReductionSumAnnuli}, we have finally
\begin{equation}
|G(\alpha+i\beta;T_N)|\leq C_0\times\sum\limits_{I\in[|0,N-1|]}C^{\#(I)}\left|\overline{G^{I}}(\alpha+i\beta;T_N)\right|.
\end{equation}
This finishes the proof of Lemma~\ref{lem:RenewalBound}.
\end{proof}

\subsection{Convergence and analyticity of the limit}
In order to prove the analyticity of the limit, it is convenient to consider a slightly modified version of regularization of local Liouville correlation function in the following way:
\begin{equation}
\mathcal{G}(\alpha+i\beta;T_N)=\mathbb{E}\left[e^{(\alpha+i\beta)B_{T_N}-\frac{(\alpha+i\beta)^2}{2}T_N}e^{-\mu\int_{0}^{\infty}e^{\gamma(B_r-Qr)}Z_r dr}\right].
\end{equation}
This function is well-defined and analytic in $(\alpha,\beta)$ for all $N$ as long as $|\beta|<Q-\alpha$.

We will show that $\mathcal{G}$ and $G$ are exponentially close as $N$ goes to infinity. More precisely:
\begin{lemm}[Modified regularization of local Liouville correlation function]\label{lem:ModifiedG}
We claim the following properties on $\mathcal{G}(\alpha+i\beta;T_N)$ and $G(\alpha+i\beta;T_N)$ in the region
\begin{equation}
\mathcal{R}^{T}_{loc}=\mathcal{R}_{loc}\cap\{\alpha<Q-\frac{\gamma}{2}\}.
\end{equation}
\begin{enumerate}
	\item There exists some constants $C,\eta>0$ locally uniform in $(\alpha,\beta)$ such that
	\begin{equation}
	\left|\mathcal{G}(\alpha+i\beta;T_N)-G(\alpha+i\beta;T_N)\right|\leq Ce^{-\eta N};
	\end{equation}
	\item There exists some constants $C,\eta>0$ locally uniform in $(\alpha,\beta)$ such that
	\begin{equation}
	\left|G(\alpha+i\beta;T_{N+1})-G(\alpha+i\beta;T_N)\right|\leq Ce^{-\eta N};
	\end{equation}
	\item There exists some constants $C,\eta>0$ locally uniform in $(\alpha,\beta)$ such that
	\begin{equation}
	\left|\mathcal{G}(\alpha+i\beta;T_{N+1})-\mathcal{G}(\alpha+i\beta;T_N)\right|\leq Ce^{-\eta N}.
	\end{equation}
\end{enumerate}
\end{lemm}

Let us first show how this can be used to finish the proof of Theorem~\ref{th:LocalStoppingTime}.
\begin{proof}[Proof of Theorem~\ref{th:LocalStoppingTime}]
By property (3) in Lemma~\ref{lem:ModifiedG}, local uniform convergence of $\mathcal{G}(\alpha+i\beta;T_N)$ yields that the limit of $\mathcal{G}(\alpha+i\beta;T_N)$ as $N$ goes to infinity is analytic in $(\alpha,\beta)$ in the region
\begin{equation}
\mathcal{R}^{T}_{loc}=\mathcal{R}_{loc}\cap\{\alpha<Q-\frac{\gamma}{2}\}.
\end{equation}

Furthurmore, by (1) of Lemma~\ref{lem:ModifiedG}, the limit is the same as the limit of $G(\alpha+i\beta;T_N)$ when $N$ goes to infinity. This proves that the limit
\begin{equation}
G^{T}(\alpha+i\beta)\coloneqq \lim_{N\to\infty}G(\alpha+i\beta;T_N)
\end{equation}
is well-defined and analytic in $\mathcal{R}^{T}_{loc}$.
\end{proof}

We now study Lemma~\ref{lem:ModifiedG}.
\begin{proof}[Proof of Lemma~\ref{lem:ModifiedG}]
The proof uses the same techniques as in the boundedness arguments.

Let us start with (2). We apply the same technique as before using the cutting annuli. The proof is almost identical with Lemma~\ref{lem:RenewalBound}, except that the set of cutting annuli over which we sum must contain the time segment with indice $N$, which gives rise to an exponential factor $e^{-\frac{N\gamma}{4q}}$.

More precisely, let us consider
\begin{align}
&G(\alpha+i\beta;T_{N+1})-G(\alpha+i\beta;T_N)\\
={}&\mathbb{E}\left[\prod_{i=0}^{N-1}\left(e^{i\beta (B_{T_{i+1}}-B_{T_{i}})+\frac{\beta^2}{2}(T_{i+1}-T_{i})}e^{-\mu e^{-\gamma i}M_{i}}\right)\times e^{i\beta (B_{T_{N+1}}-B_{T_{N}})+\frac{\beta^2}{2}(T_{N+1}-T_{N})}\left(1-e^{-\mu e^{-\gamma N}M_{N}}\right)\right]\nonumber.
\end{align}

We now apply the same argument with the indicators: in particular we add the indicator
\begin{equation}
e^{-\mu e^{-\gamma N}M_{N}}=e^{-\mu e^{-\gamma N}(M_N\wedge e^{\frac{\gamma N}{2}})}+\left(e^{-\mu e^{-\gamma N}M_N}-e^{-\mu e^{-\frac{\gamma N}{2}}}\right)\mathbf{1}_{\{M_N> e^{\frac{\gamma N}{2}}\}}
\end{equation}
to control the term with $M_N$.

For the part with the small indicator, using the same method as in boundedness Lemma~\ref{lem:Boundedness} we obtain the bound
\begin{equation}
Ce^{-\frac{\gamma N}{2}}
\end{equation}
which is of the desired form.

Now we control the part with the other indicator. That is
\begin{equation}
\left|\mathbb{E}\left[\prod_{i=0}^{N-1}\left(e^{i\beta (B_{T_{i+1}}-B_{T_{i}})+\frac{\beta^2}{2}(T_{i+1}-T_{i})}e^{-\mu e^{-\gamma i}M_{i}}\right)\mathbf{1}_{\{M_N> e^{\frac{\gamma N}{2}}\}}\left(e^{-\mu e^{-\gamma N}M_N}-e^{-\mu e^{-\frac{\gamma N}{2}}}\right)\right]\right|.
\end{equation}
This quantity has been studied before using cutting annuli, it contributes at most
\begin{equation}
Ce^{-\frac{\gamma N}{4q}},
\end{equation}
thus the proof of claim (2).

For (1), the proof is similar to (2) except that we replace the time segment $[T_N,T_{N+1}]$ by $[T_N,\infty]$, on which the drift is $-Q$ instead of $-(Q-\alpha)$. However, one can check that the rest of the proof goes in the same way as in (2).

Finally, (3) is a direct consequence of (1) and (2).
\end{proof}

\begin{rema}[Extention of stopping time method]
One can ask if we can define probabilistically the analytic continuation of local Liouville correlation function to $\mathcal{R}_{loc}$ only using this stopping time method. It is indeed possible modulo a slight modification by looking at lower-than-1 fractional moment of the GMC measure for real $\alpha$ in the disk $\mathbb{D}$. By the study of the generalized Seiberg bound (see \cite{david2016liouville,huang2018liouville}) some small moments of this GMC measure always exist as long as $\alpha<Q$, and one can adapt easily appropriately elements in the above proof to this case. Nonetheless, the nature of the stopping time method yields convergence in a somewhat weaker sense than the martingale method, although their limits coincide and yield the correct analytic continuation of (local) Liouville correlation function in $\mathcal{R}^{T}_{loc}\cap\mathcal{R}^{M}_{loc}$.
\end{rema}

\section{Proof of the main theorem}\label{sec:ProofGeneral}
To pass from the local versions Theorem~\ref{th:LocalMartingale} and Theorem~\ref{th:LocalStoppingTime} to the main Theorem~\ref{th:MainTheorem}, one uses the independence property of radial decomposition Remark~\ref{rem:Independence}. Recall that we are interested in studying the limit as $\mathbf{t}$ goes to infinity of
\begin{equation}
G(\bm{\alpha+i\beta},\mathbf{z};\mathbf{t})\coloneqq\mathbb{E}\left[\prod_{j=1}^{n}e^{(\alpha_j+i\beta_j)X_{r_j}(z_j)-\frac{(\alpha_j+i\beta_j)^2}{2}t_j}M_{\gamma}(C_t)^{-s}\right]
\end{equation}
where $r_j=e^{-t_j}$, $s=\frac{\sum_j(\alpha_j+i\beta_j)-2Q}{\gamma}$ and $C_t=\mathbb{C}-B(0,e^{-t})$. Remember that under the Seiberg bound, we have 
$\alpha_j<Q$ for all $j$ and $\Re(s)>0$.

Without lost of generosity, we can suppose $z_1=0$ and $|z_j|\geq 2$ for $j\neq 1$: one can get this configuration by applying a deterministic conformal map if necessary. Now we fix all $t_j$ for $j\neq 1$ and study the limit as $t_1$ goes to infinity: we write
\begin{equation}
M_\gamma(\mathbb{C})=M_\gamma(\mathbb{D})+M_\gamma(\mathbb{C}\backslash\mathbb{D})
\end{equation}
and remark that the second summand is independent of $\{X_{r_1}(0)\}_{r_1\leq 1}$. In view of Remark~\ref{rem:Independence}, one verifies that all the calculations in the local case can be done in this general case (since other Brownian motions are also independent of $\{X_{r_1}(0)\}_{r_1\leq 1}$ and do not enter in the calculation) as soon as $\Re(s)>0$ (which implies $\Re(s)+1>0$). This is similar to the local case with some technical modifications, we provide a detailed sketch of proof in Appendix~\ref{app:NonLocal}.

Thus, the limit when $t_1$ goes to infinity when $\{|\beta_1|<Q-\alpha_1\}$ is well-defined and analytic. One can successively apply the same procedure to $t_2,\dots,t_n$ and define the limit function $G(\bm{\alpha+i\beta};\mathbf{z})$ in this manner. This limit is analytic in the region $\mathcal{R}$ since it is the uniform local limit of analytic functions by Hartog's theorem: this completes our proof of Theorem~\ref{th:MainTheorem}.

\begin{rema}[Beyond the region $\mathcal{R}$]
It is natural to study the same question beyond the region $\mathcal{R}$. In view of the DOZZ formula \cite{dorn1992correlation,zamolodchikov1996conformal,kupiainen2017integrability}, it is reasonable to believe that the $n$-point correlation function should admit an analytic continuation to the whole complex set $\mathbb{C}^n$ as a meromorphic function, i.e. analytic modulo some poles.

It seems however that our approaches are limited to the region $\mathcal{R}$. With the martingale method, this limitation stems from the fact that, although we use a strong freezing estimate for the real part $\alpha$, our control on the complex argument part is still poor. Indeed, we always perform at some stage of our proof the following crude estimate
\begin{equation}
\forall t\in\mathbb{R}_+,\quad \left|e^{i\beta B_t}\right|\leq 1
\end{equation}
which does not take into account of the angular compensation between different phases of $B_t$. The freezing estimate for the real part translates to the condition $|\beta|<Q-\alpha$.

In the same way, with the stopping-time method, we need the condition
\begin{equation}
\mathbb{E}\left[e^{\frac{\beta^2}{2}T_1}\right]<\infty
\end{equation}
which translates again to $|\beta|<Q-\alpha$.

It is curious that both methods yield the same constraint and indicate that the region $\mathcal{R}$ is the maximal set on which analytic continuation of Liouville correlation function is possible with the current techniques for attacking this problem. It seems that at this stage that a rigorous justification for a purely probabilistic definition of the analytic continuation of $n$-point Liouville correlation functions on the Riemann sphere beyond the region $\mathcal{R}$, using directly the original approach of \cite{david2016liouville}, remains a mathematical challenge.
\end{rema}

\appendix
\section{Some facts on renewal process}\label{app:renewal}
We follow notations from Proposition~\ref{prop:renewalprop} and we use the language of renewal process. Given $t>0$, define the renewal process
\begin{equation}
N(t)=\sup\{n : T_n<t\}
\end{equation}
and the renewal function
\begin{equation}
m(t)=\mathbb{E}[N(t)].
\end{equation}

It is a classical fact that the residual time $R_{T}(t)$ at time $t$ admits the following probability density expression:
\begin{equation}
\phi(x,t)=f(t+x)+\int_{0}^{t}f(x+u)m'(t-u)du
\end{equation}
where $f_{T_1}(x)$ is the probability density of the variable $T_1$, given in Equation~\eqref{eq:T1density}. To show the last claim of Proposition~\ref{prop:renewalprop} we study the tail of $\phi(x,t)$ as $x$ goes to infinity. Since for large $x$, the density function $f(x)$ is decreasing, we have
\begin{equation}
\phi(x,t)=f(t+x)+\int_{0}^{t}f(x+u)m'(t-u)du\leq f(x)+f(x)\int_{0}^{t}m'(t-u)du\leq Cf(x)
\end{equation}
in such a way that for every $t$, $\phi(x,t)$ has the same sub-exponential tail as $f$. Since
\begin{equation}
\mathbb{E}\left[e^{\frac{\beta^2}{2}T_1}\right]<\infty
\end{equation}
as long as $|\beta|<Q-\alpha$, we have that for all $t>0$, if $|\beta|<Q-\alpha$ then
\begin{equation}
\mathbb{E}\left[e^{\frac{\beta^2}{2}R_{T}(t)}\right]<\infty.
\end{equation}

\section{Generalized freezing estimate}\label{app:freezing}
\begin{lemm}[Generalized freezing estimate]\label{lem:GeneralFreezing}
For $\sum\limits_{i}\alpha_i>Q$, $x_i\in B(0,\epsilon)$ and $\mu>0$,
\begin{equation}
\mathbb{E}\left[\exp\left(-\mu\int_{|x|>\epsilon}\prod_{i}\frac{1}{|x-x_i|^{\gamma\alpha_i}}M_\gamma(d^2x)\right)\right]\leq C\epsilon^{\frac{1}{2}(\sum\limits_{i}\alpha_i-Q)^2}
\end{equation}
with the constant $C$ independent of $\epsilon$ when $\epsilon$ is small enough.
\end{lemm}
\begin{proof}
Without loss of generosity let us suppose $X$ is correlated as
\begin{equation}
\mathbb{E}[X(x)X(y)]=\ln\frac{1}{|x-y|}
\end{equation}
in the unit disk $x,y\in\mathbb{D}$, since other cases can be reduced to this one by Kahane's inequality Lemma~\ref{lem:KahaneConvexity}. It is sufficient to prove the following statement:
\begin{equation}
\mathbb{E}\left[\exp\left(-\mu\int_{2\epsilon<|x|<1}\prod_{i}\frac{1}{|x-x_i|^{\gamma\alpha_i}}M_\gamma(d^2x)\right)\right]\leq C\epsilon^{\frac{1}{2}(\sum\limits_{i}\alpha_i-Q)^2}
\end{equation}
Since for $2\epsilon<|x|<1$ and $|x_i|<\epsilon$,
\begin{equation}
|x-x_i|\leq 2|x|,
\end{equation}
it is sufficient to prove with $\alpha=\sum\limits_{i}\alpha_i$,
\begin{equation}
\mathbb{E}\left[\exp\left(-\mu\int_{2\epsilon<|x|<1}\frac{1}{|x|^{\gamma\alpha}}M_\gamma(d^2x)\right)\right]\leq C\epsilon^{\frac{1}{2}(\alpha-Q)^2}.
\end{equation}
This is the case for the classical freezing estimate.
\end{proof}

\section{Estimates and proof in the non-local case}\label{app:NonLocal}
Following notations in Section~\ref{sec:ProofGeneral}, consider
\begin{equation}\label{eq:ZeroModeIntegral}
G(\bm{\alpha+i\beta},\mathbf{z};\mathbf{t})\coloneqq\mathbb{E}\left[\prod_{j=1}^{n}e^{(\alpha_j+i\beta_j)X_{r_j}(z_j)-\frac{(\alpha_j+i\beta_j)^2}{2}t_j}M_{\gamma}(C_t)^{-s}\right]
\end{equation}
where we fix all components $t_2,\dots,t_n$ and vary the component $t_1$. Remember that
\begin{equation}
s=\frac{\sum_{j=1}^{n}(\alpha_j+i\beta_j)-2Q}{\gamma}
\end{equation}
in such a way that $\Re(s)>0$ since we suppose the Seiberg bound. For simplicity we also suppose $z_1=0$ and $|z_i|>2$ for $i\neq 1$.

\subsection{Martingale method}
Here the assumptions on the component $t_1$ are
\begin{equation}
\alpha_1>Q-\gamma;\quad |\beta_1|<Q-\alpha_1.
\end{equation}

The martingale method works in similar manner as in Section~\ref{sec:MartingaleLocal}. Namely, we apply Itô calculus to the complex martingale
\begin{equation}
e^{(\alpha_1+i\beta_1)X_{r_1}-\frac{(\alpha_1+i\beta_1)^2}{2}t_1}
\end{equation}
we obtain, similarly to Equation~\eqref{eq:MartingaleIto},
\begin{equation}
\left|\frac{\partial G}{\partial t_1}\right|\leq\frac{\Re(s)\prod_{j=2}^{n}\left(e^{\frac{\beta_j^2}{2}t_j}\right)}{2\pi}\int_{0}^{2\pi}e^{(\gamma\alpha-2)t_1}\mathbb{E}\left[e^{\frac{\beta^2}{2}t_1}\left(\int_{C_{\{t_2,\dots,t_n\}}\backslash B(0,e^{-t_1})}\frac{1}{|x-e^{-t_1}e^{i\theta}|^{\gamma^2}}\frac{M_\gamma(d^2x)}{|x|^{\gamma\alpha}}\right)^{-\Re(s)-1}\right]d\theta
\end{equation}
where we denoted by $B_{t_1}$ the Brownian motion $X_{r_1}$. Since the (generalized) freezing estimate is also valid for negative moments (and it is local estimate around $0$), we can finish the proof in the same way as in Section~\ref{sec:MartingaleLocal}.

\subsection{Stopping time method}
Here the assumptions on the component $t_1$ are
\begin{equation}
\alpha_1<Q-\frac{\gamma}{2};\quad |\beta_1|<Q-\alpha_1.
\end{equation}

It is more convenient to go back to the original definition of the Liouville correlation function in this case. Consider the (regularized) integral over the zero-mode $c$,
\begin{equation}
C(\bm{\alpha+i\beta},\mathbf{z};\mathbf{t})\coloneqq\mathbb{E}\left[\int_{\mathbb{R}}e^{(\sum_j(\alpha_j+i\beta_j)-2Q)c}\prod_{j}e^{(\alpha_j+i\beta_j)X(z_j)}e^{-\mu e^{\gamma c}\int_{C_{\mathbf{t}}}e^{\gamma X}}dc\right]
\end{equation}
where $X$ is the GFF in the metric $\mathbf{g}$ as defined in Section~\ref{sec:GFFSetup}.

We fix all other components $t_2,\dots,t_n$ and perform Girsanov on the real part $\alpha_1$ of the component $t_1$. Using the same notation as before, we can write the regularized correlation function $C$ as
\begin{equation}
\int_{\mathbb{R}}e^{\left(\sum\limits_{j=1}^{n}(\alpha_j+i\beta_j)-2Q\right)c}\mathbb{E}\left[\prod_{j=2}^{n}e^{(\alpha_j+i\beta_j)X(z_j)}e^{-\mu e^{\gamma c}M_{\gamma}(C'_{\mathbf{t}})}\prod_{i=0}^{N-1}\left(e^{i\beta(B_{T_{i+1}}-B_{T_i})+\frac{\beta^2}{2}(T_{i+1}-T_{i})}e^{-\mu e^{\gamma c}e^{-\gamma i}M_i}\right)\right]dc
\end{equation}
with $C'_{\mathbf{t}}=\mathbb{C}\backslash\mathbb{D}-\bigcup\limits_{j=2}^{n}B(z_j,e^{-t_j})$ and $M_\gamma$ the GMC measure associated to $X$. The Brownian motion $B$, the stopping times $T$ and the measures $M$ are the same as in Section~\ref{sec:StoppingLocal}, defined with respect to the component $t_1$.

Our goal here is to prove that the whole integral decays at exponential speed (with respect to $N$) with coefficients locally uniform in $(\alpha,\beta)$ in the stopping time region for the component $t_1$. For this we adapt the proof in the local case with more careful control of the coefficients in the estimates. In particular, we study how the coefficients vary with respect to the zero-mode $c$: in other words, we study more in detail the proof of Lemma~\ref{lem:ModifiedG} with the parameter $c$.

In the following $\Omega$ will denote some non-empty open ball of $C'_{\mathbf{t}}$ such that the correlation of $X$ in $\Omega$ and in $\mathbb{D}$ is smaller than $\epsilon$: one verifies that this is always possible using the covariance kernel in Section~\ref{sec:GFFSetup}. We study the finite difference of the function
\begin{equation}
H(c,N)\coloneqq\mathbb{E}\left[\prod_{j=2}^{n}e^{(\alpha_j+i\beta_j)X(z_j)}e^{-\mu e^{\gamma c}M_{\gamma}(C'_{\mathbf{t}})}\prod_{i=0}^{N-1}\left(e^{i\beta(B_{T_{i+1}}-B_{T_i})+\frac{\beta^2}{2}(T_{i+1}-T_{i})}e^{-\mu e^{\gamma c}e^{-\gamma i}M_i}\right)\right]
\end{equation}
as $N$ goes to infinity. Our goal is to obtain an upper bound of the form
\begin{equation}
\int_{\mathbb{R}}e^{\left(\sum\limits_{j=1}^{n}\alpha_j-2Q\right)c}|H(c,N)-H(c,N+1)|dc\leq Ce^{-\eta N}
\end{equation}
for some $\eta>0$ and $C$ locally uniform in $(\alpha,\beta)$, independent of other parameters. This estimate yields exponential convergence of the function $C(\bm{\alpha+i\beta},\mathbf{z};\mathbf{t})$ along the stopping time sequence, i.e. with respect to $N$, and proves the analyticity of the limit in $t_1$. The main question is whether we can have a such a constant $C$ after the integration over the zero-mode.

To this end, we look at
\begin{equation}
H(c,N)-H(c,N+1)
\end{equation}
and by using independence properties of $B$, write it as
\begin{equation}
\mathbb{E}\left[\prod_{j=2}^{n}e^{(\alpha_j+i\beta_j)X(z_j)}e^{-\mu e^{\gamma c}M_{\gamma}(C'_{\mathbf{t}})}\prod_{i=0}^{N-1}\left(\cdots\right)\left(e^{i\beta(B_{T_{N+1}}-B_{T_N})+\frac{\beta^2}{2}(T_{N+1}-T_{N})}\left(1-e^{-\mu e^{\gamma c}e^{-\gamma N}M_N}\right)\right)\right].
\end{equation}
As before, we will add indicators and expand the product. Namely, write for all $0\leq i\leq N$,
\begin{equation}
e^{-\mu e^{\gamma c}e^{-\gamma i}M_i}=e^{-\mu e^{\gamma c}e^{-\gamma i}(M_i\wedge e^{\frac{\gamma i}{2}})}+\left(e^{-\mu e^{\gamma c}e^{-\gamma i}M_i}-e^{-\mu e^{\gamma c}e^{-\frac{\gamma i}{2}}}\right)\mathbf{1}_{\{M_i>e^{\frac{\gamma i}{2}}\}}.
\end{equation}
We sketch now how the bounds will change with respect to $c$: we will use the same notations as in the local case.

Remember that in the expansion, we will have two parts after the cutting annuli reduction: the part with large indicators where the upper bound comes from the term
\begin{equation}
\mathbf{1}_{\{M_i>e^{\frac{\gamma i}{2}}\}}
\end{equation}
which we bound in absolute value using Kahane's decorrelation inequality; and the part with small indicators where the upper bound comes from the term
\begin{equation}
1-e^{-\mu e^{\gamma c}e^{-\gamma i}(M_i\wedge e^{\frac{\gamma i}{2}})}
\end{equation}
which we also bound in absolute value using Taylor expansion.

Let us now suppose that the large indicators are taken on a $(B,\delta)$-cutting annuli $I\subset[0,N]$.

\emph{Large indicators.}
Under the assumption on $\Omega$, the same decorrelation applies using Kahane-Slepian's inequality. Following the same notations as in the proof of Lemma~\ref{lem:KahaneDecorrelation}), consider
\begin{equation}
\overline{H^{I}}(c;T_N)=\mathbb{E}\left[\prod_{j=2}^{n}e^{(\alpha_j+i\beta_j)X(z_j)}\times e^{-\mu e^{\gamma c}M_{\gamma}(\Omega)}\mathbf{1}_{I\leftrightarrow (B,\delta)}\prod_{i\in I}\left(e^{\frac{\beta^2}{2}(T_{i+1}-T_{i})}\mathbf{1}_{\{M_i>e^{\frac{\gamma i}{2}}\}}\right)\right].
\end{equation}
We first write with Hölder inequality (with same $(p,q)$ as in Lemma~\ref{lem:KahaneDecorrelation}):
\begin{equation}
\overline{H^{I}}(c;T_N)\leq \prod_{j=2}^{n}e^{\frac{\beta^2}{2}t_j}\times C^{\#(I)}\mathbb{E}\left[e^{-p\mu e^{\gamma c}M_{\gamma}(\Omega)}\right]^{\frac{1}{p}}\mathbb{E}\left[\mathbf{1}_{I\leftrightarrow (B,\delta)}\prod_{i\in I}\mathbf{1}_{\{M_i>e^{\frac{\gamma i}{2}}\}}\right]^{\frac{1}{q}}.
\end{equation}
Performing the same argument as in Lemma~\ref{lem:KahaneDecorrelation} we arrive at the same upper bound estimate (remember that this part does not depend on $\mu$ or $c$) with an additional factor
\begin{equation}
\mathbb{E}\left[e^{-p\mu e^{\gamma c}M_{\gamma}(\Omega)}\right]^{\frac{1}{p}}.
\end{equation}

\emph{Small indicators.}
By Corollary~\ref{cor:SmallIndicatorMu}, we can bound the small indicators contribution by
\begin{equation}
C(\mu)\left(1+C^{\frac{c}{2}}\right)
\end{equation}
if it does not contain the indice $N$, and by
\begin{equation}
C(\mu)\left(1+C^{\frac{c}{2}}\right)\times e^{\gamma c}e^{-\frac{\gamma N}{2}}
\end{equation}
if it contains the indice $N$, for some finite constant $C(\mu)$ independent of $c$ and $N$.

\emph{Zero-mode integration.}
From the discussion above, for each $c$, the bound changes in the large indicator estimate by
\begin{equation}
\mathbb{E}\left[e^{-p\mu e^{\gamma c}M_{\gamma}(\Omega)}\right]^{\frac{1}{p}}
\end{equation}
and in the small indicator estimate, on the scale of at most
\begin{equation}
e^{(\gamma+\sigma)c}
\end{equation}
for large positive $c$ and some small positive deterministic $\sigma$.

Now we have to make sure that integrating over $c$, these factors are still bounded by some constant. Hence we look at, with $p>1$,
\begin{equation}
\int_{\mathbb{R}}e^{\left(\frac{\sum_{j=1}^{n}\alpha_j-2Q}{\gamma}\right)c}e^{(\gamma+\sigma)(c\vee 0)}\mathbb{E}\left[e^{-p\mu e^{\gamma c}M_{\gamma}(\Omega)}\right]^{\frac{1}{p}}dc
\end{equation}
and ask if this integral is convergent. It suffices to look at the integration over the positive $c$ since the expectation part is bounded by $1$. But for all $\rho>0$,
\begin{equation}
\int_{\mathbb{R}_+}e^{\rho c}\mathbb{E}\left[e^{-p\mu e^{\gamma c}M_{\gamma}(\Omega)}\right]^{\frac{1}{p}}dc
\end{equation}
can be bounded by
\begin{equation}
\left(\int_{\mathbb{R}_+}e^{-qc}dc\right)^{\frac{1}{q}}\left(\int_{\mathbb{R}_+}e^{p(\rho+1)c}\mathbb{E}\left[e^{-p\mu e^{\gamma c}M_{\gamma}(\Omega)}\right]dc\right)^{\frac{1}{p}}
\end{equation}
by Hölder (where $q$ is the conjugate of $p$), and this bound is finite independent of $c$: it can be bounded by a fixed negative moment of the GMC measure $M_{\gamma}(\Omega)$.

Finally, one can perform the same combinatorics using the cutting annuli to conclude. We thus obtain $(\alpha,\beta)$-locally uniform exponential decay in $N$ as expected.

\bibliographystyle{alpha}
\bibliography{reference}

\end{document}